\documentclass[12pt]{amsart}
\usepackage{}
\usepackage{amsmath}
\usepackage{amsfonts}
\usepackage{amssymb}
\usepackage[all,cmtip]{xy}           
\usepackage{xspace}
\usepackage{bbding}
\usepackage{txfonts}
\usepackage[shortlabels]{enumitem}
\usepackage{ifpdf}
\ifpdf
  \usepackage[colorlinks,final,backref=page,hyperindex]{hyperref}
\else
  \usepackage[colorlinks,final,backref=page,hyperindex,hypertex]{hyperref}
\fi
\usepackage{tikz}
\usepackage[active]{srcltx}

\makeatletter

\newtheorem{theorem}{Theorem}[section]
\newtheorem{prop}[theorem]{Proposition}
\newtheorem{lemma}[theorem]{Lemma}
\newtheorem{coro}[theorem]{Corollary}
\newtheorem{prop-def}{Proposition-Definition}[section]

\theoremstyle{definition}
\newtheorem{defn}[theorem]{Definition}

\newtheorem{remark}[theorem]{Remark}
\newtheorem{exam}[theorem]{Example}

\newcommand{\nc}{\newcommand}

\nc{\delete}[1]{{}}
\nc{\mmargin}[1]{}

\nc{\mlabel}[1]{\label{#1}}  
\nc{\mcite}[1]{\cite{#1}}  
\nc{\mref}[1]{\ref{#1}}  
\nc{\meqref}[1]{\eqref{#1}} 
\nc{\mbibitem}[1]{\bibitem{#1}} 

\delete{
\nc{\mlabel}[1]{\label{#1}  
{\hfill \hspace{1cm}{\bf{{\ }\hfill(#1)}}}}
\nc{\mcite}[1]{\cite{#1}{{\bf{{\ }(#1)}}}}  
\nc{\mref}[1]{\ref{#1}{{\bf{{\ }(#1)}}}}  
\nc{\meqref}[1]{\eqref{#1}{{\bf{{\ }(#1)}}}} 
\nc{\mbibitem}[1]{\bibitem[\bf #1]{#1}} 
}

\nc{\cmttv}{}
\nc{\twrba}{that is a twisted Rota-Baxter algebra\xspace}

\nc{\shadow}{phantom\xspace}
\nc{\Shadow}{Phantom\xspace}
\nc{\shad}{\theta}
\nc{\tforall}{\text{ for all }}
\nc{\ddiff}{unit-modified differential\xspace}
\nc{\ddiffrey}{differential Reynolds\xspace}
\nc{\oplin}{operator linear\xspace}
\nc{\fopav}{F_A(\Omega,V)}
\nc{\tay}{power\xspace}
\nc{\tw}{twisted\xspace}
\nc{\wtd}{weighted\xspace}
\nc{\wtddiff}{modified differential\xspace}
\nc{\intdiff}{integro-differential\xspace}
\nc{\wtdintdiff}{modified integro-differential\xspace}
\nc{\wt}{weight\xspace}
\nc{\wte}{\lambda}
\nc{\coa}{{\frak O\frak C\frak A}} 
\nc{\oa}{{\frak O\frak A}} 
\nc{\ctoa}{{\frak T\frak O\frak C\frak A}} 
\nc{\toa}{{\frak T\frak O\frak A}} 
\nc{\om}{\mathfrak{OM}_A}
\nc{\ocm}{\mathfrak{OCM}_A}

\nc{\frakO}{\mathfrak{O}}
\nc{\frakD}{\mathfrak{D}}
\nc{\oam}{\frakO_A\frakM}
\nc{\ocam}{\frakC\frakC\frakO_A\frakM}
\nc{\moverline}[1]{\mathrm{cl}({#1})}
\nc{\compopgen}{compatible complete operated $A$-module\xspace}
\nc{\compopgens}{compatible complete operated $A$-modules\xspace}
\nc{\shortneed}{{\rm CCO$_A$M}\xspace}
\nc{\shortneeds}{{\rm CCO$_A$Ms}\xspace}

\nc{\ocmra}{compatible weighted Reynolds algebra\xspace}
\nc{\ocmras}{compatible weighted Reynolds algebras\xspace}
\nc{\ocmdra}{compatible \ddiffrey algebra\xspace}
\nc{\ocmdras}{compatible \ddiffrey algebras\xspace}

\nc{\algw}{\frakS}	
\nc{\opw}{\Delta}	
\nc{\ralgw}{\frakI}	
\nc{\ropw}{\Pi}		
\nc{\sett}{\calt}	
\nc{\algt}{\frakT}	
\nc{\opt}{\Lambda}	
\nc{\rsett}{\cale}	
\nc{\ralgt}{\frakE}	
\nc{\ropt}{\Gamma}	

\nc{\tprod}{\veebar}
\nc{\graftprod}{grafting product\xspace}
\nc{\extnop}{extension operator\xspace}
\nc{\frakT}{\mathfrak{T}}

\nc{\frakJ}{\mathfrak{J}}
\nc{\frakI}{\mathfrak{I}}

\nc{\eva}{\mathrm{eva}}

\nc{\mfraka}{\tau}

\nc{\trsha}{{\mathfrak F}_{\rm{RRB}}}

\nc{\reypr}{\circledast}	

\nc{\lc}{\lfloor} \nc{\rc}{\rfloor}
\nc{\free}[1]{\overline{#1}}
\nc{\Id}{\mathrm{Id}}
\nc{\lra}{\longrightarrow}
\nc{\hra}{\hookrightarrow}

\nc{\mrba}{matching Rota-Baxter algebra\xspace}
\nc{\mrbas}{matching Rota-Baxter algebras\xspace}
\nc{\rsha}{\sha^{\rm rel}}

\newcommand{\bk}{{\mathbf{k}}}

\nc{\vep}{\varepsilon}
\nc{\bin}[2]{ (_{\stackrel{\scs{#1}}{\scs{#2}}})}  
\nc{\binc}[2]{(\!\! \begin{array}{c} \scs{#1}\\
    \scs{#2} \end{array}\!\!)}  
\nc{\bincc}[2]{  ( {\scs{#1} \atop
    \vspace{-1cm}\scs{#2}} )}  
\nc{\bs}{\bar{S}}
\nc{\ra}{\longleftarrow}
\nc{\ot}{\otimes}
\nc{\rar}{\rightarrow}
\nc{\dar}{\downarrow}
\nc{\dap}[1]{\downarrow \rlap{$\scriptstyle{#1}$}}
\nc{\defeq}{\stackrel{\rm def}{=}}
\nc{\dis}[1]{\displaystyle{#1}}
\nc{\dotcup}{\ \displaystyle{\bigcup^\bullet}\ }
\nc{\hcm}{\ \hat{,}\ }
\nc{\hts}{\hat{\otimes}}
\nc{\hcirc}{\hat{\circ}}
\nc{\lleft}{[}
\nc{\lright}{]}
\nc{\curlyl}{\left \{ \begin{array}{c} {} \\ {} \end{array}
    \right .  \!\!\!\!\!\!\!}
\nc{\curlyr}{ \!\!\!\!\!\!\!
    \left . \begin{array}{c} {} \\ {} \end{array}
    \right \} }
\nc{\longmid}{\left | \begin{array}{c} {} \\ {} \end{array}
    \right . \!\!\!\!\!\!\!}
\nc{\ora}[1]{\stackrel{#1}{\rar}}
\nc{\ola}[1]{\stackrel{#1}{\la}}
\nc{\scs}[1]{\scriptstyle{#1}} \nc{\mrm}[1]{{\rm #1}}
\nc{\dirlim}{\displaystyle{\lim_{\longrightarrow}}\,}
\nc{\invlim}{\displaystyle{\lim_{\longleftarrow}}\,}
\nc{\dislim}[1]{\displaystyle{\lim_{#1}}} \nc{\colim}{\mrm{colim}}
\nc{\mvp}{\vspace{0.3cm}} \nc{\tk}{^{(k)}} \nc{\tp}{^\prime}
\nc{\ttp}{^{\prime\prime}} \nc{\svp}{\vspace{2cm}}
\nc{\vp}{\vspace{8cm}}
\nc{\modg}[1]{\!<\!\!{#1}\!\!>}
\nc{\intg}[1]{F_C(#1)}
\nc{\lmodg}{\!<\!\!}
\nc{\rmodg}{\!\!>\!}
\nc{\cpi}{\widehat{\Pi}}
\nc{\sha}{{\mbox{\cyr X}}}  
\nc{\ssha}{{\mbox{\cyrs X}}} 
\nc{\tsha}{{\mbox{\cyrt X}}}
\nc{\shai}{{\stackrel{\ra}{\sha}}}
\nc{\sshai}{{\stackrel{\ra}{\ssha}}}
\nc{\shpr}{\diamond}    
\nc{\labs}{\mid\!}
\nc{\rabs}{\!\mid}

\nc{\dr}{\frakR}
\nc{\cdr}{\frakC\frakR}
\nc{\er}{\frakW\frakR}
\nc{\cer}{\frakC\frakW\frakR}
\nc{\fdr}{\frakF\frakC\frakR}
\nc{\fer}{\frakF\frakC\frakW\frakR}
\nc{\DR}{\overline{\frakR}}
\nc{\CDR}{\overline{\frakC\frakR}}
\nc{\FDR}{\overline{\frakF\frakC\frakR}}
\nc{\ER}{\overline{\frakW\frakR}}
\nc{\CER}{\overline{\frakC\frakW\frakR}}
\nc{\FER}{\overline{\frakF\frakC\frakW\frakR}}

\font\cyr=wncyr10
\font\cyrs=wncyr7
\font\cyrt=wncyr5

\nc{\ann}{\mrm{ann}}
\nc{\Aut}{\mrm{Aut}}
\nc{\can}{\mrm{can}}
\nc{\Cont}{\mrm{Cont}}
\nc{\rchar}{\mrm{char}}
\nc{\cok}{\mrm{coker}}
\nc{\dtf}{{R-{\rm tf}}}
\nc{\dtor}{{R-{\rm tor}}}

\nc{\Div}{{\mrm Div}}
\nc{\End}{\mrm{End}}
\nc{\Ext}{\mrm{Ext}}
\nc{\Fil}{\mrm{Fil}}
\nc{\Fr}{\mrm{Fr}}
\nc{\Frob}{\mrm{Frob}}
\nc{\Gal}{\mrm{Gal}}
\nc{\GL}{\mrm{GL}}
\nc{\Hom}{\mrm{Hom}}
\nc{\hsr}{\mrm{H}}
\nc{\hpol}{\mrm{HP}}
\nc{\id}{\mrm{id}}
\nc{\im}{\mrm{im}}
\nc{\incl}{\mrm{incl}}
\nc{\length}{\mrm{length}}
\nc{\mforall}{\quad \text{for all }}
\nc{\mchar}{\rm char}
\nc{\mpart}{\mrm{part}}
\nc{\ql}{{\QQ_\ell}}
\nc{\qp}{{\QQ_p}}
\nc{\rank}{\mrm{rank}}
\nc{\rcot}{\mrm{cot}}
\nc{\rdef}{\mrm{def}}
\nc{\rdiv}{{\rm div}}
\nc{\rtf}{{\rm tf}}
\nc{\rtor}{{\rm tor}}
\nc{\res}{\mrm{res}}
\nc{\SL}{\mrm{SL}}
\nc{\Spec}{\mrm{Spec}}
\nc{\tor}{\mrm{tor}}
\nc{\Tr}{\mrm{Tr}}
\nc{\tr}{\mrm{tr}}

\nc{\bfk}{{\bf k}}
\nc{\bfone}{{\bf 1}}
\nc{\bfzero}{{\bf 0}}
\nc{\Diff}{\mathbf{Diff}}
\nc{\FMod}{\mathbf{FMod}}
\nc{\Int}{\mathbf{Int}}
\nc{\Mon}{\mathbf{Mon}}
\nc{\remarks}{\noindent{\bf Remarks: }}
\nc{\Rep}{\mathbf{Rep}}
\nc{\Rings}{\mathbf{Rings}}
\nc{\Sets}{\mathbf{Sets}}

\nc{\BA}{{\mathbb A}}   \nc{\CC}{{\mathbb C}}
\nc{\DD}{{\mathbb D}}   \nc{\EE}{{\mathbb E}}
\nc{\FF}{{\mathbb F}}   \nc{\GG}{{\mathbb G}}
\nc{\HH}{{\mathbb H}}   \nc{\LL}{{\mathbb L}}
\nc{\NN}{{\mathbb N}}   \nc{\PP}{{\mathbb P}}
\nc{\QQ}{{\mathbb Q}}   \nc{\RR}{{\mathbb R}}
\nc{\TT}{{\mathbb T}}   \nc{\VV}{{\mathbb V}}
\nc{\ZZ}{{\mathbb Z}}   \nc{\TP}{\widetilde{P}}

\nc{\cala}{{\mathcal A}}    \nc{\calc}{{\mathcal C}}
\nc{\cald}{\mathcal{D}}     \nc{\cale}{{\mathcal E}}
\nc{\calf}{{\mathcal F}}    \nc{\calg}{{\mathcal G}}
\nc{\calh}{{\mathcal H}}    \nc{\cali}{{\mathcal I}}
\nc{\call}{{\mathcal L}}    \nc{\calm}{{\mathcal M}}
\nc{\caln}{{\mathcal N}}    \nc{\calo}{{\mathcal O}}
\nc{\calp}{{\mathcal P}}    \nc{\calr}{{\mathcal R}}
\nc{\cals}{{\mathcal S}}    \nc{\calt}{{\mathcal T}}
\nc{\calw}{{\mathcal W}}    \nc{\calx}{{\mathcal X}}
\nc{\CA}{\mathcal{A}}

\nc{\fraka}{{\mathfrak a}}
\nc{\frakb}{{\mathfrak b}}
\nc{\frakc}{{\mathfrak c}}
\nc{\frakd}{{\mathfrak d}}
\nc{\frake}{{\mathfrak e}}
\nc{\fraki}{{\mathfrak i}}
\nc{\frakj}{{\mathfrak j}}
\nc{\frakk}{{\mathfrak k}}
\nc{\frakB}{{\mathfrak B}}
\nc{\frakC}{{\mathfrak C}}
\nc{\frakE}{{\mathfrak E}}
\nc{\frakF}{{\mathfrak F}}
\nc{\frakG}{{\mathfrak G}}
\nc{\frakm}{{\mathfrak m}}
\nc{\frakM}{{\mathfrak M}}
\nc{\frakp}{{\mathfrak p}}
\nc{\frakR}{{\mathfrak R}}
\nc{\frakS}{{\mathfrak S}}
\nc{\frakA}{{\mathfrak A}}
\nc{\frakW}{{\mathfrak W}}
\nc{\fraku}{{\mathfrak u}}
\nc{\frakv}{{\mathfrak v}}
\nc{\frakw}{{\mathfrak w}}
\nc{\frakx}{{\mathfrak x}}
\nc{\fraky}{{\mathfrak y}}

\nc{\vsa}{\vspace{-.1cm}}
\nc{\vsb}{\vspace{-.2cm}}
\nc{\vsc}{\vspace{-.3cm}}
\nc{\vsd}{\vspace{-.4cm}}
\nc{\vse}{\vspace{-.5cm}}

\nc{\ynr}[1]{\textcolor{blue}{\underline{Yunnan:}#1 }}

\nc{\lir}[1]{\textcolor{red}{\underline{Li:}#1 }}

\nc{\rgr}[1]{\textcolor{orange}{\underline{Richard:}#1 }}

\begin{document}

\title[Reynolds algebras from Volterra integrals and complete shuffle product]{
Generalized Reynolds algebras from Volterra integrals and their free construction by complete shuffle product	}

\author{Li Guo}
\address{Department of Mathematics and Computer Science, Rutgers University, Newark, NJ 07102, United States}
\email{liguo@rutgers.edu}

\author{Richard Gustavson}
\address{Department of Mathematics, Farmingdale State College, Farmingdale, NY 11735, United States}
\email{gustavr@farmingdale.edu}

\author{Yunnan Li}
\address{School of Mathematics and Information Science, Guangzhou University, Guangzhou 510006, China}
\email{ynli@gzhu.edu.cn}

\date{\today}

\begin{abstract}
	This paper introduces algebraic structures for Volterra integral operators with separable kernels, in the style of differential algebra for derivations and Rota-Baxter algebra for operators with kernels dependent solely on a dummy variable. We demonstrate that these operators satisfy a generalization of the algebraic identity defining the classical Reynolds operator, which is rooted in Reynolds's influential work on fluid mechanics.
	To study Volterra integral operators and their integral equations through this algebraic lens, particularly in providing a general form of these integral equations, we construct free objects in the category of algebras equipped with generalized Reynolds operators and the associated differential operators, termed differential Reynolds algebras. Due to the cyclic nature of the Reynolds identity, the natural rewriting rule derived from it does not terminate. To address this challenge, we develop a completion for the underlying space, where a complete shuffle product is defined for the free objects. We also include examples and applications related to Volterra integral equations.
\end{abstract}

\subjclass[2020]{
45D05, 
17B38, 
12H05, 
16S10, 
47G20, 
47A62. 
}

\keywords{Volterra integral equation, Volterra operator,  Reynolds algebra, Rota-Baxter algebra, differential algebra, completion, shuffle product, free object}

\maketitle

\vspace{-1cm}

\tableofcontents

\vspace{-1cm}

\allowdisplaybreaks

\section{Introduction}
This paper studies the algebraic linear identities satisfied by separable Volterra integral operators. Motivated by giving an algebraic framework for the integral equations of such integral operators, the free objects in the corresponding algebraic categories are constructed by completing the shuffle product. More precisely, this paper provides
\begin{enumerate}
\item[(i)] a precise algebraic framework for studying separable Volterra integral equations via differential Reynolds algebras.  This generalization of the standard Reynolds algebra allows for an algebraic version of the separable Volterra integral operator that takes into account the kernel of the integral, as well as a modified differential operator that together produce a generalized algebraic version of the Fundamental Theorem of Calculus;
\item[(ii)] the construction of the free objects in the category of commutative differential Reynolds algebras.  Due to the recursive nature of the Reynolds identity, this requires taking the completion of the underlying space, with corresponding complete shuffle product.  In the case of separable Volterra integral equations, this completion is made explicit as the space of formal power series with corresponding integral and differential operators.
\end{enumerate}

\subsection{Algebraic operators from analysis}
A remarkable phenomenon in the development of mathematics is that interesting algebraic operators have mostly been introduced from analytic studies, providing examples, stimulations and motivations to the general study of algebraic operators that satisfy operator identities. In the other direction, the study of these algebraic operators provides a general framework to understand the analytic phenomena. 

A case study is the relationship between differential algebras and differential equations. Almost a century ago, from his study of differential equations, Ritt~\mcite{Rit} introduced the algebraic notion of {\bf differential operators}, by extracting the Leibniz rule
$$ d(xy)=d(x)y+xd(y)
$$
satisfied by the derivation in analysis. In the following decades, the theory has been fully developed to include branches such as differential Galois theory, differential algebraic geometry and differential algebraic groups~\mcite{Kol,PS}, with broad connections to areas in mathematics and mathematical physics~\mcite{FLS,MS,Wu}.
On the other hand, differential algebra has powerful applications back to differential equations as shown already in the work of Ritt. On a concrete and fundamental level, the construction of free differential algebras provides a uniform framework to consider all differential equations, leading to such developments as differential decomposition algorithms to help solve systems of algebraic differential equations and solutions to the parameter identifiability problem for input-output equations~\mcite{BLOP,OPT}.

Algebraic abstractions for integral operators have appeared in various forms due to the different notions of integrations. The most well-known and also the simplest is the {\bf Rota-Baxter operator of weight zero}:
$$ P(x)P(y)=P(xP(y))+P(P(x)y),$$
modeled after the integration by parts formula for the simple integral operator $I(f)(x)\coloneqq \int_a^xf(t)\,dt$. In general, a {\bf Rota-Baxter operator of weight $\lambda$}, defined by
$$ P(x)P(y)=P(xP(y))+P(P(x)y)+\lambda P(xy)$$
for a fixed scalar $\lambda$, arose from the probability study of G. Baxter~\mcite{Bax} in 1960 and pursued further by Atkinson, Cartier, and especially Rota, in the following decades~\mcite{At,Ca,Ro}. The algebraic abstraction with the varying weight leads to broad applications in combinatorics, number theory and quantum field theory~\mcite{CK1,Gub,Ro0}.

The construction of free Rota-Baxter algebras of weight zero by tensor powers and the shuffle product~\mcite{GK1,Ro} underlies the structure and multiplication of iterated integrals with simple kernels (depending only on the dummy variable; see Definition~\mref{de:volop}), which can be tracked back to the work of K.-T. Chen~\mcite{Ch,Ch2} in differential geometry and have found broad applications from multiple zeta values and Hodge theory to rough paths~\mcite{FH,Ha,IKZ,Ra}. The free objects also give a general form of Volterra equations with simple kernels~\mcite{GGL}. 

Furthermore, the First Fundamental Theorem of Calculus abstracts to the notions of {\bf differential Rota-Baxter algebras} and {\bf integro-differential algebras}, applied to the algebraic study of boundary problems for linear ordinary differential equations~\mcite{GGR,GK3,GRR,HRR,RaR,RR}. This algebraic theory has produced applications such as finding solutions to integro-differential equations and contributions to parameter estimation techniques for nonlinear dynamical systems~\mcite{Bav,BKLPPU,BLR,QR2,RRTB}.

\subsection{Volterra operators and Reynolds operators}
More general than the simple integral operator $I(f)(x)=\int_a^xf(t)\,dt$ is the {\bf Volterra operator}
\[ P_K(f)(x)\coloneqq P_{K,a}(f)(x) \coloneqq  \int_a^x K(x,t)f(t)\,dt,  \vsa\]
where $K(x,t)$ is a fixed function, called the kernel.
Such operators and the Fredholm integral operators are the two primary types of integral operators and integral equations, named after the two founders of the theory of integral equations.
See~\mcite{Trb,Vo,Wa,Ze} for the role the operators play in the general theory of integral equations, and \mcite{GGL} for a recent algebraic study of Volterra operators of which the present paper is a sequel. As noted in~\mcite{GGL}, the operator $P_K$ satisfies the Rota-Baxter operator only in the very special case when $K(x,t)$ depends on $t$ only.

Our {\em first goal} in this study is to establish the operator identity satisfied by $P_K$ when $K(x,t)$ is separable in the sense that $K(x,t)=k(x)h(t)$ for single variable functions $k(x)$ and $h(t)$ (instead of the more general form of $\sum_{i}k_i(x)h_i(t)$). Then the Volterra operator is called {\bf separable}.
It is here that the Reynolds operator enters the picture.

The {\bf Reynolds operator of weight $\lambda$} is defined by the operator identity
\begin{equation*}
	R(x)R(y)=R(xR(y))+R(R(x)y)-\lambda R(R(x)R(y))
	\mlabel{eq:rey}
\end{equation*}
for a given scalar $\lambda$. The operator (when $\lambda=1$) first arose from O. Reynolds' famous work on turbulence theory in fluid mechanics~\mcite{Re} and attracted interests of R.~Birkhoff and G.-C.~Rota~\mcite{Bi,RoR,Ro1} in the 1960s.  See~\mcite{AS,CHK,Uc,ZtGG,ZtGG0} for recent studies related to algebra, combinatorics and invariant theory.

An important example of the Reynolds operator is the separable Volterra operator $P_K$ when $K(x,t)=e^{x-t}$~\mcite{RoR}.
In this paper we generalize the notion of a Reynolds algebra to have more general weights, defined by the {\bf \wtd Reynolds identity}
\begin{equation}
R(f)R(g) = R(R(f)g) + R(fR(g)) - R(\lambda R(f) R(g))
\mlabel{eq:grey}
\end{equation}
for a fixed element $\lambda$ in the underlying algebra, instead of only a scalar.
We will show (Theorem~\mref{thm:sep}) that, when $K(x,t)$ is separable, the Volterra operator $P_K$ is a weighted Reynolds operator of weight $\lambda$.
To determine the weight $\lambda$, associated to this \wtd Reynolds operator is a \wtd differential operator $D_K$ which serves as the left inverse of the \wtd Reynolds operator, generalizing the First Fundamental Theorem of Calculus. Then the weight $\lambda$ is $D_K(1)$. The two operators $D_K$ and $P_K$, together with their coupling, form the key notions of this paper, called the {\bf \ddiff operator}, the {\bf \wtd Reynolds operator} and the {\bf \ddiffrey operator} respectively (see Definition~\mref{de:rrb}).
As a general phenomenon, these operators also capture the algebraic properties of a differential or Rota-Baxter operator of weight zero after pre- or post-composed by a left multiplication operator by a fixed element (Theorem~\mref{th:composing}). 

\vsb
\subsection{Free objects for \ddiffrey algebras}
The {\em second goal} of this paper is to construct the free objects in the category of commutative \ddiffrey algebras. Here as in the case of differential equations and Volterra integral equations with simple kernels mentioned above, our motivation is to use the free objects to provide a general framework for Volterra integral equations with separable kernels and to express such equations in terms of iterated integrals.

Free objects in an algebraic category are often obtained by regarding the defining identity of the algebraic structure as a rewriting system from which the irreducible elements form a linear basis of the free object. This is the approach for the previously discussed differential, Rota-Baxter and several other operators. 

However, since the left-hand side of the defining identity of the weighted Reynolds operator in Eq.~\meqref{eq:grey} also appears on the right-hand side, the resulting cyclic rewriting system, 
\begin{equation} R(f)R(g)\mapsto R(R(f)g) + R(fR(g)) - R(\lambda R(f) R(g)),
 \vsa
\mlabel{eq:rrr}
\end{equation}
by taking the left-hand side as the leading term, will not terminate. In previous studies of the free Reynolds algebras (in the classical case)~\mcite{ZtGG,ZtGG0}, a different rewriting system was extracted from the Reynolds identity in order to bypass this difficulty.
In the more recent paper~\mcite{GGL}, the cyclic rewriting system in Eq.~\meqref{eq:rrr} is converted to a finite rewriting system at the presence of additional restrictions. 

In this paper, this cyclic difficulty is tackled directly and is used as an opportunity to understand the topological aspect of rewriting systems~\mcite{Che,Kob}.
In order for the infinite process to make sense, we introduce a complete space for which the leading term reappears after the rewriting but with a higher order for the completion (Definition~\mref{def:com}).
Repeated applications of the rewriting lead to a formal series in the complete space. This can be regarded as an algebraic counterpart of the method of successive substitution in solving integral equations~\cite[\S3.7]{Wa}.

In a similar context, complete Rota-Baxter algebras were studied in~\mcite{GK2}, generalizing the process of completing the polynomial algebra to the power series algebra. In contrast to the Rota-Baxter case, a complete topological vector space is indispensable to define the multiplication in the free objects for Reynolds algebras.
Many years ago, Rota made the observation that the Reynolds operator is an ``infinitesimal analog" of the Rota-Baxter operator~\mcite{Ro1}. In confirmation to this observation, we show that
the shuffle product that defines the multiplication in a free Rota-Baxter algebra expands to a complete shuffle product that defines the multiplication in a free Reynolds algebra.
 \vsc
\subsection{Outline of the paper}

Here is a summary of the paper.

In Section~\mref{sec:diffrey}, we first introduce a class of linear operators arising from Volterra operators with separable kernels, including the \ddiff operator, the \wtd Reynolds operator and the \ddiffrey operator. The \ddiff operator is a left inverse of the \wtd Reynolds operator with a special weight, generalizing the algebraic formulation of the First Fundamental Theorem of Calculus as given in~\mcite{GK3,RaR}. We then give their realizations as Volterra operators with separable kernels.  In particular, a Volterra integral operator is a \wtd Reynolds operator with a corresponding \ddiff operator (Theorem~\mref{thm:sep}). So \ddiffrey algebras provide an algebraic context to study separable Volterra integral equations. Properties of these new operators are obtained.

Section~\mref{sec:freeobj} has its goal of constructing free \ddiffrey algebras. To overcome the difficulty caused by the cyclic property of the \wtd Reynolds identity, we begin with a general notion of an operated module that is complete with respect to a filtration which is compatible with the linear operator. Then the usual shuffle product defined on tensor-power polynomials is given a complete version on tensor-power series, which then is applied to construct the multiplication in the free objects in a subcategory of the \ddiffrey algebras, called the category of commutative \ocmdras (Theorem~\mref{thm:freedrey}).
We finally apply the free construction to the analytic setting of Volterra integral equations by examples. 

\smallskip

\noindent
{\bf Notations.} In this paper, we fix a ground field $\bfk$ of characteristic 0. All the objects
under discussion, including vector spaces, algebras and tensor products, are taken over $\bfk$ unless otherwise specified.
An algebra is assumed to be a unitary associative algebra.

\section{Algebraic structures from Volterra integral equations}
\mlabel{sec:diffrey}

Studies on integral operators and integral equations motivate us to introduce a new class of algebraic operators, generalizing the well-known concepts of the differential operator, Reynolds operator and Rota-Baxter operator.

\subsection{Algebraically defined operators}

We present the various algebraic structures that have arisen from the study of Volterra integral operators, by listing the existing ones first and then introducing the new ones. Together, they will be the main algebraic structures treated in this study.
\begin{defn}\label{de:rbr}
	\begin{enumerate}
		\item  		\label{it:op}
		A {\bf pointed algebra} is an algebra $R$ together with a fixed element $\lambda\in R$.
		\item An {\bf operated algebra} \mcite{Gop} is a pair $(R,P)$ consisting of an algebra $R$ and a linear operator $P$ on $R$.
		\item 	\mlabel{it:rb}
		For a fixed $\lambda \in \bfk$, a {\bf Rota-Baxter algebra of weight} $\lambda$ is an algebra $R$ together with a linear operator $P:R \to R$ satisfying
\begin{equation}
\mlabel{eq:rb}
		P(f)P(g) = P(fP(g)) + P(P(f)g) + \lambda P(fg) \quad \text{for all } f,g \in R.
\end{equation}
\item \mlabel{it:rey}
A {\bf Reynolds algebra} is an algebra $R$ together with a linear operator $P:R \to R$ satisfying
\begin{equation}
\mlabel{eq:rey2}
P(f)P(g) = P(fP(g)) + P(P(f)g) - P(P(f)P(g)) \quad \text{for all } f,g \in R.
\end{equation}
\item \mlabel{it:diff}
For a given $\lambda\in \bfk$, a {\bf differential algebra of weight $\lambda$}~\mcite{GK3} is an algebra $R$ equipped with a linear operator $d$ satisfying
\begin{equation} \mlabel{eq:wtdif}
	d(xy)=d(x)y+xd(y)+\lambda d(x)d(y) \quad \text{for all } x, y\in R.	
\end{equation}
\item \mlabel{it:intdiff}
An {\bf \intdiff algebra  of weight $\lambda$}~\mcite{GRR,RR} is a differential algebra $(R,d)$ of weight $\lambda$ with a linear operator $p$ on $R$
satisfying $dp=\id_R$ and the {\bf \intdiff identity}
\begin{equation}\mlabel{intdiff}
p(d(x)))p(d(y))=p(d(x))y+xp(d(y))-p(d(xy))\quad \text{for all }  x,y\in R.
\end{equation}
\end{enumerate}
\end{defn}

A submodule $I$ of an operated algebra $(R,P)$ is called an {\bf operated ideal} if $I$ is an (algebraic) ideal of $R$ such that $P(I)\subseteq I$. Then the quotient module $R/I$ has the induced operated algebra structure. The same notions apply to any of the special classes of operated algebras listed above and those to be introduced later.

Each of the above classes of algebras forms a category with the morphisms defined as follows.
\begin{enumerate}
	\item A homomorphism from a pointed algebra $(A,\lambda_A)$ to $(B,\lambda_B)$ is an algebra homomorphism $f:A\to B$ such that $f(\lambda_A)=\lambda_B$.
	\item For operated algebras $(R_1,P_1)$ and $(R_2,P_2)$, an algebra homomorphism $\varphi:R_1\rar R_2$ is called an {\bf operated algebra homomorphism} if $\varphi P_1=P_2 \varphi$. \item
	The same definition also applies when operated algebra is replaced by any of the special cases in Items~\meqref{it:rb} -- \meqref{it:intdiff}.
\end{enumerate}

We now introduce new types of operated algebras which also have realizations as Volterra integral operators.
Note that, while traditionally, the weight $\lambda$ of a differential or Rota-Baxter operator is referring to an element in the base ring $\bfk$; here the notion has been generalized to allow $\lambda$ to be in the algebra $R$ itself.

\begin{defn}
	\begin{enumerate}
		\item 
		A {\bf \wtddiff algebra of weight $\wte\in R$} is a pointed algebra $(R,\wte)$ with a linear operator $D:R\rar R$, called a {\bf \wtddiff operator}\footnote{When $\wte$ is a scalar, the operator is one of the differential type operators in~\mcite{GSZ} and the term \wtddiff operator is used in~\mcite{PZGL}. This terminology is justified by Lemma~\mref{lem:twistdiff}.} satisfying
		\begin{equation}\mlabel{wda}
			D(xy)=D(x)y+xD(y)-x\wte y,\quad  x,y\in R.
		\end{equation}		
\item A unital \wtddiff algebra $(R,D,\wte)$  with $\wte=D(1_R)$ is called a {\bf \ddiff algebra}.
So  $D$ satisfies
	\begin{equation}\mlabel{Dda}
	D(xy)=D(x)y+xD(y)-xD(1_R)y,\quad  x,y\in R.
	\end{equation}
\item
A {\bf Reynolds algebra of weight $\wte$} or simply a {\bf  \wtd Reynolds algebra} is a triple $(R,P,\wte)$ consisting of a pointed algebra $(R,\wte)$ and a linear operator $P:R\rar R$, called a {\bf \wtd Reynolds operator} (of weight $\wte$), satisfying
		\begin{equation}\mlabel{Gra}
			P(x)P(y)=P(P(x)y)+P(xP(y))-P(P(x)\wte P(y)),\quad  x,y\in R.
		\end{equation}	
\item
A {\bf modified differential Reynolds algebra of weight $\wte$}, which we will simply call a {\bf differential Reynolds algebra of weight $\wte$},  is a quadruple $(R,D,P,\wte)$ such that $(R,D,\wte)$ is a modified differential algebra of weight $\wte$, $(R,P,\wte)$ is a Reynolds algebra of weight $\wte$, and $D P=\id_R$ holds.
		
\item A {\bf \ddiff Reynolds algebra}, which we will simply call a {\bf \ddiffrey algebra}, is a unital differential Reynolds algebra $(R,D,P,D(1_R))$ of weight $D(1_R)$. So both Eq.~\meqref{Dda} and the following equations hold.
		\begin{equation}
			P(x)P(y)=P(P(x)y)+P(xP(y))-P(P(x)D(1_R)P(y)),
			\quad	D P=\id_R,
			\quad  x,y\in R.
			\mlabel{Dra}	
		\end{equation}		
\item 
A {\bf \wtdintdiff algebra} $(R,D,P)$ is a \ddiff algebra $(R,D)$ with a linear operator $P$ on $R$, called an {\bf integration}, satisfying $DP=\id_R$ and the {\bf \wtdintdiff identity}
\begin{equation}\mlabel{wtdint}
		\begin{split}
&P(D(x)))P(D(y))\\
  &=P(D(x))y+xP(D(y))-P(D(xy))-P(D(1_R))\big(xy-P(D(xy))\big),\quad  x,y\in R.
		\end{split}
\end{equation}
	\end{enumerate}
	\mlabel{de:rrb}	
\end{defn}

\begin{remark} 
\begin{enumerate}
\item
Note the difference between a differential algebra of weight $\lambda$ and a \wtddiff algebra as indicated in the last terms of their defining equations \meqref{eq:wtdif} and \meqref{wda}.
\item 
If a \wtddiff algebra $(R,D,\wte)$ of weight $\wte$ is unital,  then taking $x=y=1_R$ in Eq.~\eqref{wda} yields $\wte=D(1_R)$. Thus being a unital \wtddiff algebra is the same as being a \ddiff algebra. For the same reason, being a unital modified differential Reynolds algebra of weight $\wte$ is the same as being a \ddiff Reynolds algebra. 
Nevertheless, we will retain the terms of the \ddiff algebra and the \ddiff Reynolds algebra since they are the main notions in this study.
\item 
It will be shown in Corollary~\mref{co:funrey} that the differential Reynolds algebra and modified integro-differential algebra both arise from separable Volterra integral operators. The focus of this paper will be on differential Reynolds algebras only. Modified integro-differential algebra is a variation of the integro-differential algebra~\mcite{GRR,RaR,RR} and will be treated in more detail separately. 
\end{enumerate}
\end{remark}

Each of the above classes of algebras also forms a category with the corresponding morphisms defined as follows.
\begin{enumerate}
\item
For \wtddiff algebras $(R_1,D_1,\wte_1)$ and $(R_2,D_2,\wte_2)$, an algebra homomorphism $\varphi:R_1\rar R_2$ is called a {\bf \wtddiff algebra homomorphism} if $\varphi(\wte_1)=\wte_2$ and $\varphi D_1=D_2 \varphi$. Note that if $R_1$ and $R_2$ are \ddiff algebras, then $\varphi(\lambda_1)=\lambda_2$ follows immediately from $\varphi D_1=D_2 \varphi$ since $\varphi(1_{R_1}) = 1_{R_2}$.
\item
For \wtd Reynolds algebras $(R_1,P_1,\wte_1)$ and $(R_2,P_2,\wte_2)$, an algebra homomorphism $\varphi:R_1\rar R_2$ is called a {\bf \wtd Reynolds algebra homomorphism} if $\varphi(\wte_1)=\wte_2$ and $\varphi P_1=P_2 \varphi$.
\item
For \ddiffrey algebras $(R_1,D_1,P_1)$ and $(R_2,D_2,P_2)$, a \ddiff algebra homomorphism $\varphi:R_1\rar R_2$ is called a  {\bf \ddiffrey algebra homomorphism} if $\varphi P_1=P_2 \varphi$.	The same applies to \wtdintdiff algebras. 
\end{enumerate}

\begin{exam}
The identity operator $\id_A$ on an algebra $A$ is a \ddiff operator, making any algebra a \ddiff algebra.
Also, any algebra $A$ has a natural \ddiffrey algebra structure $(A,\id_A,\id_A)$.
\end{exam}

\nc{\multop}{L}

For a fixed element $\lambda$ in an algebra $R$, let 
$$\multop_\lambda:R\to R, x\mapsto \lambda x,$$ 
denote the operator of left multiplication by $\lambda$. 
It is well known that post-composing $\multop_\lambda$ to a differential operator of weight zero is still such an operator. 
Likewise, pre-composing a Rota-Baxter operator of weight zero 
by $\multop_\lambda$ is easily checked to be again a Rota-Baxter operator of weight zero~\mcite{GGL}.
We now give the algebraic structures from composing $\multop_\lambda$ in the opposite directions. 

\begin{theorem} 
\begin{enumerate}
\item\label{t1}
Let $(R,d)$ be a differential algebra of weight $0$ with a central element $\wte$. 
Then $(R,d \multop_\lambda,d(\wte))$ is a \wtddiff algebra of weight $d(\wte)$.
\item\label{t2}
Let $(R,d,P)$ be a unital 
\intdiff algebra of weight $0$ with an invertible central element $\wte$.
Then $(R,d \multop_{\lambda^{-1}}, \multop_\lambda P)$ is a \wtdintdiff algebra, and also a \ddiffrey algebra.
\end{enumerate}
\mlabel{th:composing}
\end{theorem}
\begin{proof}
\eqref{t1} For any $x,y\in R$, we have
\begin{align*}
d\multop_\lambda(x)y+xd\multop_\lambda(y)-xd(\wte)y&=d(\wte x)y+xd(\wte y)-xd(\wte)y\\
&\stackrel{\eqref{eq:wtdif}}{=} d(\wte x )y+ x\big(d(\wte)y+\wte d(y)\big)-xd(\wte)y\\
&=d(x\wte)y+x\wte d(y)\\
&\stackrel{\eqref{eq:wtdif}}{=} d(x\wte y)=d(\wte xy)=d\multop_\lambda(xy).
\end{align*}
So $d\multop_\lambda$ is a \wtddiff operator of weight $d(\wte)$.

\eqref{t2} 
For clarity, we tentatively denote $D:=d\multop_{\lambda^{-1}}$ and $\Pi:=\multop_\lambda P$ for the proof. 
Since $D(1_R)=d(\wte^{-1})$, we see that $(R,D)$ is a \ddiff algebra by Item~\eqref{t1}.   
By~\cite[Theorem~2.5]{GRR}, the evaluation map $e\coloneqq \id_R-Pd$
is an algebra homomorphism. 
Correspondingly, we define
\begin{equation}\mlabel{eq:e}
	E:=\id_R-\Pi D=\wte e\wte^{-1}.
	\end{equation}
Then for any $x,y\in R$, we obtain
$$E(xy)=\wte e(\wte^{-1}xy)=\wte e(\wte)e(\wte^{-1}x)e(\wte^{-1}y)= E(1_R)^{-1}E(x)E(y),$$
since $E(1_R)^{-1}=\lambda^{-1}e(\lambda^{-1})^{-1}=\wte^{-1}e(\wte)$.This is just Eq.~\eqref{wtdint} thanks to Eq.~\meqref{eq:e}, showing that $(R,D,\Pi)$ is a \wtdintdiff algebra. 

On the other hand, substituting $\Pi(x)$ for $x$ and $\Pi(y)$ for $y$ in Eq.~\eqref{wtdint}, we get 
$$\Pi(x)\Pi(y)=\Pi(x)\Pi(y)+\Pi(x)\Pi(y)-\Pi(D(\Pi(x)\Pi(y)))-\Pi(D(1_R))\big(\Pi(x)\Pi(y)-\Pi(D(\Pi(x)\Pi(y)))\big).$$
Namely, $E(1_R)\Big(\Pi(x)\Pi(y)-\Pi\big(D(\Pi(x)\Pi(y)\big))\Big)=0$. Since $E(1_R)$ is invertible and $D$ is a \wtddiff operator, we have
$$\Pi(x)\Pi(y)=\Pi(D(\Pi(x)\Pi(y)))\stackrel{\eqref{Dda}}{=}\Pi(\Pi(x)y)+\Pi(x\Pi(y))-\Pi(\Pi(x)D(1_R)\Pi(y)).$$
Thus $(R,D,\Pi)$ is a \ddiffrey algebra.
\end{proof}

\subsection{Realizations from integral operators}
We show that all the above algebraic notions find natural realizations as Volterra integral operators.
We consider continuous functions $C(I)$ on an open interval $I$ in $\RR$ or an open square in $\RR^2$. In fact we usually take $I=\RR$ or $\RR^2$ for simplicity.

\begin{defn}
	Let $I\subseteq \RR$ be an open interval. Fix  $K(x,t) \in C(I^2)$.
	\begin{enumerate}
		\item
		A {\bf Volterra (integral) operator} is a linear operator $P_K\coloneqq P_{K,\,a}:C(I) \to C(I)$ defined by 
		\vsa
		\[
		P_{K,\,a}(f)(x) = \int_a^x K(x,t)f(t)\,dt,
		\vsa
		\]
		for some $a \in I$. Here $K$ is called the {\bf kernel}.
		\item
		A kernel $K(x,t)$ and the corresponding Volterra operator are called {\bf separable} if it can be decomposed as $K(x,t) = k(x)h(t)$ for some functions $k$ and $h$ in $C(I)$.
		\item
		A kernel $K(x,t)$ and the corresponding Volterra operator are called {\bf \shadow} if it is a function of only the variable $t$ of integration (i.e. the ``dummy" variable). It is the special case of a separable kernel $K(x,t)=k(x)h(t)$ when $k(x)$ is a constant.
		\item
		An integral equation is called {\bf separable} (resp. {\bf \shadow}) if all the integral operators in the equation are separable (resp. \shadow) and share the same lower limit.
	\end{enumerate}
	\mlabel{de:volop}
\end{defn}
More details about Volterra integral equations with separable kernels can be found in e.g. \cite{Ze}.

For the simple case when $K(x,t)$ is \shadow, as noted before Theorem~\mref{th:composing}, $P_K$ is still a Rota-Baxter operator of weight zero. 
On the other hand, when the kernel $K(x,t)$ is indeed a function of $x$ in addition to $t$, the Volterra integral operator $P_K$ is no longer a Rota-Baxter operator, as shown in the simple counterexample.

\begin{exam}	
	Let $K(x,t)=x$ and $f=g=1$, then the integral operator $P_K$ on $C(\RR)$ with lower limit $a=0$ gives
	$P_K(f)(x)\,P_K(g)(x)=x^4$ and
	$P_K(fP_K(g))(x)+P_K(P_K(f)g)(x)=
	\frac{2}{3}x^4.$
\end{exam}	

In general, a Volterra operator with separable kernel $K(x,t) = k(x)h(t)$ is a Rota-Baxter operator only in very special circumstances, as shown in Corollary~\mref{co:seprb}.
Nevertheless, the operator is a \wtd Reynolds operator.

\begin{theorem}	\mlabel{thm:sep}
	Let $a\in I$ and let $K(x,t) = k(x)h(t)$ with $k \in C^1(I)$ and $h \in C(I)$ both zero free, where $C^1(I)$ consists of all differentiable functions over $I$. Define
	\begin{equation}
		D_K:C^1(I)\rar C(I), \quad
		D_{K}(f)(x)\coloneqq \frac{1}{h(x)}\left(\frac{f(x)}{k(x)}\right)' =\frac{k(x)f'(x)-k'(x)f(x)}{h(x)k(x)^2}.
		\mlabel{eq:Ddo}
	\end{equation}
\begin{enumerate}
\item \mlabel{it:sep1}
The operator $D_K$ is a \ddiff operator$:$
\vsa
	\begin{equation}	
		\mlabel{eq:lr}
		D_{K}(fg) = D_{K}(f)g + fD_{K}(g) - D_{K}(1)fg, \quad f, g\in C^1(I).
	\end{equation}
\item
The Volterra integral operator \vsb
	\[P_K:C(I)\rar C^1(I), P_{K}(f)\coloneqq  \int_a^xk(x)h(t)f(t)dt\]
	is a \ddiffrey operator$:$
	\begin{equation}
		P_{K}(f)P_{K}(g) = P_{K}(fP_{K}(g)) + P_{K}(P_{K}(f)g) - P_{K}(D_{K}(1)P_{K}(f)P_{K}(g)), \quad  f, g\in C(I), 
		\mlabel{eq:srb}
	\end{equation}
and $D_KP_K=\id_{C(I)}$.
\mlabel{it:sep2}
\item \mlabel{it:sep3}The operators $D_K$ and $P_K$ are also subject to the modified integro-differential identity~\eqref{wtdint}, namely, for $f,g\in C^1(I)$, we have 
\begin{equation}
\begin{split}
&P_K(D_K(f)))P_K(D_K(g))\\
  &=P_K(D_K(f))g+fP_K(D_K(g))-P_K(D_K(fg))-P_K(D_K(1))(fg-P_K(D_K(fg))).
\end{split}
 \mlabel{eq:intmid}
\end{equation}
\end{enumerate}
\end{theorem}

\begin{proof}
Define two linear maps 
$$d_h:C^1(I)\rar C(I), \quad
		d_h(f)(x)\coloneqq h(x)^{-1}f'(x),$$
and 
$$p_h:C(I)\rar C^1(I), \quad
		p_h(f)(x)\coloneqq \int_a^xh(t)f(t)dt.$$
Then $d_h$ satisfies the classical Leibniz rule for derivations.
Moreover, we have 
$d_hp_h=\id_{C(I)}$ and 
$$p_h(d_h(f))=\int_a^xf'(t)dt=f(x)-f(a),\quad f\in C^1(I),$$
so the integro-differential identity \eqref{intdiff} holds for $d_h$ and $p_h$.

Furthermore, we have 
$$D_K(f(x))=d_h(k(x)^{-1}f(x)),\quad P_K(f(x))=k(x)p_h(f(x)).$$ 
Thus $D_K=d_h \multop_{k(x)^{-1}}$ and $P_K=\multop_{k(x)} p_h$ with the notations in Theorem~\mref{th:composing}. Then the conclusions follow from the same proof as the one for Theorem~\mref{th:composing}. 
\end{proof}

The algebra $C^1(I)$ in Theorem~\mref{thm:sep} is not closed under $D_K$. To obtain a \ddiffrey algebra and also a \wtdintdiff algebra, we restrict the domain and range. 

\begin{coro}
	Let $R\coloneqq C^\infty(I)$ and $C^\infty(I^2)$ be the algebra of infinitely differentiable functions on $I\subseteq \RR$ and $I^2\subseteq \RR^2$. Then for $K(x,t)=k(x)h(t)\in C^\infty(I^2)$ zero free, the tuple $(R,D_K,P_K)$ is a \ddiffrey algebra and a \wtdintdiff algebra. \mlabel{co:funrey}
\end{coro}

We also give specific examples.
\begin{exam}
	\begin{enumerate}
\item For $K(x,t)=e^{-x+t}=e^t/e^x$, we have $D_K(1)=1$. So the operator $P_K:C(\RR)\to C^1(\RR)$ satisfies the original Reynolds identity in Eq.~\meqref{eq:rey2}~\mcite{Ro1,ZtGG}.
\mlabel{it:exp}
\item For $K(x,t)=\frac{1}{x^2+1}$, we have $D_K(1)=2x$. So the operator $P_K:C(\RR)\to C^1(\RR)$ satisfies the identity
\[P_K(f)P_K(g) = P_K(P_K(f)g)+P_K(fP_K(g))-P_K\left(2xP_K(f)P_K(g)\right).\]
\end{enumerate}
\mlabel{ex:sepex}
\end{exam}

Eq.~\meqref{eq:Ddo} specializes to
\begin{equation}
	D_K(1)=-\frac{k'(x)}{h(x)k(x)^2}=\dfrac{1}{h(x)}\left(\dfrac{1}{k(x)}\right)'.
	\mlabel{eq:done}
\end{equation}
Then we obtain the characterization for a Volterra operator to satisfy the Rota-Baxter operator.

\begin{coro}
With the same assumption as Theorem~\mref{thm:sep}, we have $D_K(1)=0$ if and only if $k(x)$ is a nonzero constant. Thus when $K$ is separable, the Volterra operator $P_K$ is a Rota-Baxter operator (of weight zero) if and only if $K$ is \shadow.
\mlabel{co:seprb}
\end{coro}

\subsection{Properties of the algebraic operators}
\mlabel{ss:algprop}
We now present more properties of the algebraic operators.

First a \wtddiff operator is a formal inverse of a \wtd Reynolds operator.
\begin{prop}
	Let $D$ be an invertible linear operator on an algebra $R$. Then $D$ is a \wtddiff operator if and only if its inverse $D^{-1}$ is a \wtd Reynolds operator.
	\mlabel{pp:diffrey}
\end{prop}
\begin{proof}
	Let $D:R\to R$ be invertible. If $D$ is a \wtddiff operator on $R$, then we have
	\begin{align*}
		D(D^{-1}(x)D^{-1}(y))&\stackrel{\eqref{wda}}{=}D(D^{-1}(x))D^{-1}(y)+D^{-1}(x)D(D^{-1}(y))
		-D^{-1}(x)\wte D^{-1}(y)\\
		&=xD^{-1}(y)+D^{-1}(x)y-D^{-1}(x)\wte D^{-1}(y).
	\end{align*}
	So
	\begin{equation} D^{-1}(x)D^{-1}(y)=D^{-1}(xD^{-1}(y)+D^{-1}(x)y-D^{-1}(x)\wte D^{-1}(y)), \quad x,y\in R,
		\mlabel{eq:wdiffinv}
	\end{equation}
	and $D^{-1}$ is a \wtd Reynolds operator.
	
	Conversely, if $D^{-1}$ is a \wtd Reynolds operator, applying $D$ to Eq.~\meqref{eq:wdiffinv} and taking $u=D^{-1}(x), v=D^{-1}(y)$ show that $D$ is a \wtddiff operator on $R$.
\end{proof}

The term \wtddiff operator is justified by the relation with the usual differential operator in the following lemma, in analog to the relation between the modified Rota-Baxter operator and the Rota-Baxter operator of weight zero~(see \mcite{Gub}). The proof is direct as in ~\cite{PZGL} where $\lambda$ is assumed to be a scalar.

\begin{lemma}
Let $\lambda$ be in the center of $R$. A linear operator $D$ on an algebra $R$ is a \wtddiff operator of weight $\lambda$ if and only if the operator $D-\lambda \id_R$ is a differential operator $($of weight $0$$)$.
\mlabel{lem:twistdiff}
\end{lemma}

\begin{remark}
	Applying Proposition~\mref{pp:diffrey} and Lemma~\mref{lem:twistdiff} to the classical case of weight $\lambda=1_R$, we recover a classical result in \mcite{Mil}: if $d$ is a derivation, then the formal inverse $P=(\id_R+d)^{-1}$ is a Reynolds operator. Conversely, if $P$ is an invertible Reynolds operator, then $d=P^{-1}-\id_R$ is a derivation.
\end{remark}

The following result shows that the operator identity~\meqref{Gra} for a \wtd Reynolds algebra $R$ is compatible with the associativity.

Let $(R,P,\wte)$ be a \wtd Reynolds algebra. Define a new multiplication
$\star$ on $R$ by
\begin{equation}
	x\star y\coloneqq P(x)y+xP(y)-P(x)\wte P(y)\quad \tforall x,y\in R.
	\mlabel{eq:derrey}
\end{equation}
Thus the Reynolds identity \meqref{Gra} is simply $P(x)P(y)=P(x\star y)$.

\begin{prop}
For a \wtd Reynolds algebra $(R,P,\wte)$, the pair $(R,\star)$ is an algebra.
If in addition, either $P(\wte)=\wte$ or
$\wte\in\bk$, then the tuple $(R,\star,P,\wte)$ also a \wtd Reynolds algebra.
\end{prop}

\begin{proof}
For $x,y,z\in R$, we have
\begin{align*}
(x\star y)\star z&=
P(x\star y)z+(x\star y)P(z)-P(x\star y)\wte P(z)\\
&\stackrel{\eqref{Gra}}{=}P(x)P(y)z+(P(x)y+xP(y)-P(x)\wte P(y))P(z)-P(x)P(y)\wte P(z)\\
&=P(x)P(y)z+P(x)yP(z)+xP(y)P(z)-P(x)\wte P(y)P(z)-P(x)P(y)\wte P(z),
\end{align*}
and
\begin{align*}
x\star(y\star z)&=P(x)(y\star z)+xP(y\star z)-P(x)\wte P(y\star z)\\
&\stackrel{\eqref{Gra}}{=}P(x)(P(y)z+yP(z)-P(y)\wte P(z))+xP(y)P(z)-P(x)\wte P(y)P(z)\\
&=P(x)P(y)z+P(x)yP(z)+xP(y)P(z)-P(x)\wte P(y)P(z)-P(x)P(y)\wte P(z),
\end{align*}
showing the associativity of $\star$.
	
We further have
	$$P(x)\star P(y)=P^2(x)P(y)+P(x)P^2(y)-P^2(x)\wte P^2(y),$$
	while
	$$P(P(x)\star y+x\star P(y)-P(x)\star \wte\star P(y))\stackrel{\eqref{Gra}}{=}P^2(x)P(y)+P(x)P^2(y)-P^2(x)P(\wte)P^2(y).$$
	Then we see that $(R,\star,P,\lambda)$ is a \wtd Reynolds algebra under either of the additional conditions stated in the proposition.
\end{proof}

\begin{remark}
\mlabel{rk:rey}
For a {\it commutative} \ddiffrey algebra $(R,D,P)$, Eq. \meqref{Dra} becomes
\begin{equation}
	P(x)P(y)=P(P(x)y)+P(xP(y))-P(D(1_R)P(x)P(y)),\quad x,y\in R.
	\mlabel{eq:Drac}
\end{equation}
The rewriting of $P(x)P(y)$ from this equality will not terminate since the left-hand side also appears on the right-hand side.
Assuming $D(1_R)=1_R$ for simplicity and repeatedly applying Eq.~\meqref{eq:Drac}, we formally have 
\begin{eqnarray*}
P(x)^2&=& 2P(xP(x))-P(P(x)^2) \\
&=& 2P(xP(x)) - 2P(P(xP(x)))+P^2(P(x)^2) \\
&=& \cdots \\
&=& 2\sum_{n=1}^\infty (-1)^{n-1} P^{n}(xP(x)).
\end{eqnarray*}
In particular, taking $P$ to be the Volterra operator $P_K(f)(u)=\int_0^uK(u,t) f(t)\,dt$ with $K(u,t)=e^{-u+t}$ as in Example~\mref{ex:sepex}, squaring the integral has the formal series expansion 
\begin{equation}
\mlabel{eq:reysq}
P_K(f)^2=2\sum_{n=1}^\infty (-1)^{n-1} P_K^{n}(fP_K(f))
\end{equation}
into iterated integrals. 

In principle, any integral expression with this kernel should also be a formal series of iterated integrals by the same idea. 
In order to make this idea rigorous and for such series to make sense, we need to work in a complete space.  
This leads to the general discussion in  Section~\mref{sec:freeobj}. Then this particular case will be revisited in Proposition~\mref{pp:powerex} and Example~\mref{ex:rey2}.
\vspace{-.2cm}
\end{remark}

\vspace{-.3cm}

\section{Free commutative \ddiffrey algebras}
\mlabel{sec:freeobj}

Due to the cyclic property of the Reynolds identity in Eqs.~\meqref{Gra} and \meqref{Dra}, the resulting rewriting leads to infinite iterations that require a completeness condition of the underlying space in order for the process to possibly converge. This especially applies to the construction of free \ddiffrey algebras. We first introduce a category with a suitable completeness condition, and then construct the free objects in this category.

\subsection{The categories of \wtd Reynolds algebras and \ddiffrey algebras}

Beginning with the category of $A$-modules, we successively build the category of operated $A$-modules, of complete operated $A$-modules with respect to a filtration, and of such complete operated $A$-modules in which the filtration is compatible with the operators. On top of the latter, we consider \wtd Reynolds algebras and \ddiffrey algebras that are compatible complete operated $A$-modules. They will serve as the category for the construction of the free \wtd Reynolds algebras and free \ddiffrey algebras.

\subsubsection{Compatible complete operated modules}

We first define compatible operated complete $A$-modules.

\begin{defn}
Fix a $\bfk$-algebra $A$.
\begin{enumerate}
\item A {\bf $\bfk$-linear operated $A$-module}, or simply an {\bf operated $A$-module}, is a pair $(R,P)$ with an $A$-module $R$ and a $\bfk$-linear operator $P$ on $R$. Note that the operator is only required to be $\bfk$-linear, not $A$-linear.
A {\bf homomorphism $f: (R,P)\to (R',P')$ of operated $A$-modules} is a homomorphism $f:R\to R'$ of $A$-modules such that $fP=P'f$.
\item
An $A$-module $R$ with a decreasing filtration $\{R_n\}_{n\geq 0}$ of $A$-submodules is called {\bf complete} if the natural $A$-linear homomorphism
$\kappa_R : R \to \varprojlim R/R_n$
is an isomorphism.
\item A {\bf complete operated $A$-module} is a triple $(R,R_n,P)$ where $(R,P)$ is an operated $A$-module and $\{R_n\}_{n\geq 0}$ is a decreasing filtration of operated $A$-submodules such that $(R,R_n)$ is a complete $A$-module.
\item
Let  $(R,R_n,P)$ be a complete operated $A$-module.
For $n\geq 0$, recursively define
\begin{equation}
	\overline{\Fil}^0_PR:=R, \quad \overline{\Fil}^n_PR:=\moverline{ A P(\overline{\Fil}^{n-1}_PR)}, \quad n\geq 1.
\mlabel{eq:complfilter}
\end{equation}
Here $\moverline{N}$ denotes the closure of an $A$-submodule $N$ in $R$ with respect to the topology given by the filtration $\{R_n\}_{n\geq 0}$.
\item
A complete operated $A$-module $(R,R_n,P)$ is said to have an {\bf operator generated filtration} if $R_n=\overline{\Fil}^n_P R,\ n\geq 0$. Then we call
$(R,R_n,P)$  a {\bf \compopgen}.
To be specific for later use, a \compopgen is a triple $(R,R_n,P)$ where
\begin{enumerate}
\item the pair $(R,P)$ is an operated $A$-module,
\item the pair $(R,R_n)$ is a complete $A$-module,
\item the above two structures are {\bf compatible} in the sense that $R_n=\overline{\Fil}^n_P R, n\geq 0$.
\end{enumerate}
\end{enumerate}
\mlabel{def:com}
\end{defn}

We give some notions on complete $A$-modules for later use.
Let $(R,R_n)$ be a filtered $A$-module. For $m>n\geq 0$, let
$$\kappa_{m,n}:R/ R_m \rar R/R_n , \quad \kappa_n:R\rar R/ R_n,$$
be the canonical projections. Then $\kappa_{m,n}\kappa_m=\kappa_n$.
Let
\begin{equation}
		\hat{R}\coloneqq \varprojlim R/R_n
	= \bigg\{(r_k+R_k)_{k\geq0}\in \prod_{k\geq0}R/R_k\,\bigg|\,r_{k+1}-r_k\in R_k,\ k\geq 0\bigg\}
	\mlabel{eq:inv_limit}
\end{equation}
be the inverse limit that gives the completion of $(R,R_n)$.
Then the natural $A$-module homomorphism $\kappa_R$ takes the form
\begin{equation}
	\kappa_R : R \to \hat{R},\ r\mapsto (r+R_k)_{k\geq 0},\quad  r\in R.
	\mlabel{eq:kr}
\end{equation}
Also let
$$\hat{\kappa}_n:\hat{R}\rar \prod_{k\geq 0} R/R_k\to R/R_n,\quad n\geq 0, $$
be the canonical projections. Then $\kappa_{m,n}\hat{\kappa}_m=\hat{\kappa}_n$ and $\hat{\kappa}_n\kappa_R=\kappa_n$.

\begin{lemma}
If $(R,R_n,P)$ is a \compopgen, then $P(R_n)\subseteq R_{n+1}$ for each $n\geq 0$.
\end{lemma}
\begin{proof}
By the definition of $\overline{\Fil}^{n+1}_P R$ in Eq.~\meqref{eq:complfilter}, we have $P(R_n)=P(\overline{\Fil}^n_P R)\subseteq \overline{\Fil}^{n+1}_P R=R_{n+1}$ for $n\geq 0$.
\end{proof}

\subsubsection{Compatible weighted Reynolds algebras and \ocmdras}
Adding a Reynolds algebra or a differential Reynolds algebra structure to a \compopgen gives us the main structures of interest.

\nc{\crac}{\mathcal{CR}_A}
\nc{\cdrac}{\mathcal{CDR}_A}

\begin{defn}
Fix a commutative pointed algebra $(A,\lambda)$. A {\bf \ocmra} over $(A,\lambda)$ is a quintuple $(R,R_n,P,\lambda_R,i_R)$ where
\begin{enumerate}
\item \mlabel{it:wr1}
$(R,\lambda_R)$ is an $(A,\lambda)$-algebra, in the sense that $R$ is an $A$-algebra of which the structure map $i_R:A\to R$ sends $\lambda$ to $\lambda_R$;
\item \mlabel{it:wr3}
with respect to the $A$-module structure on $R$ as an $A$-algebra, the triple $(R,R_n,P)$ is a \compopgen;
\item \mlabel{it:wr2}
$(R,P,\lambda_R)$ is a $\lambda_R$-\wtd Reynolds algebra.
\end{enumerate}
Let $\crac$ denote the category of \ocmras over $(A,\lambda)$ in which a morphism is defined to be a morphism for each of the three components \meqref{it:wr3}--\meqref{it:wr2}.
\end{defn}

\begin{defn}
Fix a commutative \ddiff algebra $(A,d)$. A {\bf \ocmdra} over $(A,d)$ is a quintuple $(R,R_n,D,P,i_R)$ where
	\begin{enumerate}
\item \mlabel{it:umd1}
$(R,D)$ is an $(A,d)$-\ddiff algebra, in the sense that $R$ is an $A$-algebra such that $i_R\,d=D\,i_R$ for the structure map $i_R:A\to R$. Note that $i_R$ is at the same time a homomorphism of pointed algebras $(A,d(1_A))\to (R,D(1_R))$;		
\item \mlabel{it:umd3}
with respect to the $A$-module structure on $R$ as an $A$-algebra, the triple $(R,R_n,P)$ is a \compopgen;
\item \mlabel{it:umd2}
$(R,D,P)$ is a \ddiffrey algebra.
\end{enumerate}
Let $\cdrac$ denote the category of \ocmdras over $(A,d)$ in which a morphism is defined to be a morphism for each of the three components \meqref{it:umd3}--\meqref{it:umd2}.
\end{defn}

\begin{exam}
	\mlabel{ex:comp}
By Corollary~\mref{co:funrey}, the triple $(C^\infty(\RR),D_K,P_K)$ with $K(x,t) = k(x)h(t)\in C^\infty(\RR^2)$  free of zeros, is a commutative \ddiffrey algebra.

Consider the following decreasing filtration of operated $\RR$-submodules of $(C^\infty(\RR),P_K)$,
		\begin{equation}
			C^\infty(\RR)_0=C^\infty(\RR),\quad C^\infty(\RR)_n=\{f\in C^\infty(\RR)\,|\,f^{(i)}(0)=0,\ 0\leq i\leq n-1\},\quad  n\geq 1.
			\mlabel{eq:contfil}
		\end{equation}
Note that, for the homomorphism $\kappa_R$ in Eq.~\meqref{eq:kr}, we have
		\[\ker\kappa_{C^\infty(\RR)}=\bigcap_{n\geq0} C^\infty(\RR)_n=\left\{f\in C^\infty(\RR)\,\left|\,f^{(n)}(0)=0,\,  n\geq0\right.\right\}\neq\{0\}.\]
		For example, the function
		$f(x)=\begin{cases}
			e^{-x^{-2}},&x\neq0,\\
			0,&x=0,
		\end{cases}$
		is in such an intersection. Hence,
$$\kappa_{C^\infty(\RR)}:C^\infty(\RR)\rar\varprojlim C^\infty(\RR)/C^\infty(\RR)_n$$
		is not injective and $(C^\infty(\RR),P_K)$ is not a complete operated $\RR$-module.
\end{exam}

We next give an example of a \ocmdra. It is applicable for the formal series in Eq.~\meqref{eq:reysq} of Remark~\mref{rk:rey} to converge. 
\begin{prop}
Let $\RR[[x]]$ be the formal power series algebra with its term-by-term differentiation and integration.
Let $K(x,t)=k(x)h(t)$ with $k(x)\in \RR[[x]]$ and $h(t)\in \RR[[t]]$ both invertible $($that is, with nonzero constant terms$)$. As in Theorem~\mref{thm:sep}, define
\begin{equation*}
D_K:\RR[[x]]\rar \RR[[x]], \quad
D_{K}(f)(x)\coloneqq \frac{1}{h(x)}\left(\frac{f(x)}{k(x)}\right)' =\frac{k(x)f'(x)-k'(x)f(x)}{h(x)k(x)^2}
	\mlabel{eq:Ddo1}
\end{equation*}
and
\[P_K:\RR[[x]] \rar \RR[[x]], P_{K}(f)\coloneqq  \int_0^xk(x)h(t)f(t)dt.\]
Let $A$ be an intermediate ring between $\RR$ and $\RR[[x]]$ that is $D_K$-invariant and denote $d:=(D_K)|_{A}$. Then with the filtration $\RR[[x]]_n:=x^n\RR[[x]], n\geq 0$ and the natural inclusion $i_{\RR[[x]]}:A\to \RR[[x]]$, the quintuple
$$(\RR[[x]],\RR[[x]]_n,D_K,P_K,i_{\RR[[x]]})$$
is a \ocmdra over the \ddiff algebra $(A,d)$.
\mlabel{pp:powerex}
\end{prop}

\begin{proof}
Since $A$ is $D_K$-invariant, the restriction $d$ of $D_K$ to $A$ defines a \ddiff operator
on $A$. Hence for the inclusion $i_{\RR[[x]]}:A\to \RR[[x]]$, we have $i_{\RR[[x]]} d=D_K i_{\RR[[x]]}$.
Thus $(\RR[[x]],D_K)$ is an $(A,d)$-\ddiff algebra. 
Also, Theorem~\mref{thm:sep} can be applied to elements of $\RR[[x]]$ to directly show that $(\RR[[x]],D_K,P_K)$ is a commutative \ddiffrey algebra.

It remains to check that $(\RR[[x]],\RR[[x]]_n,P_K)$ is a \compopgen.
Termwise integration gives
\begin{equation}
P_K(x^n \RR[[x]])\subseteq x^{n+1}\RR[[x]],\quad n\geq 0,
\mlabel{eq:termwiseint}
\end{equation}
and hence the filtration $\{\RR[[x]]_n\}_{n\geq0}$ is a decreasing filtration of operated $A$-submodules (actually ideals) of $(\RR[[x]],P_K)$.
Then the natural $A$-module isomorphism
		$$\kappa_{\RR[[x]]}:\RR[[x]] \rar\varprojlim \RR[[x]]/\RR[[x]]_n$$
shows that $(\RR[[x]],\RR[[x]]_n)$ is a complete operated $A$-module.

Finally for the compatibility between the operated module and the complete module structure, we will verify the equalities
\begin{equation}
\overline{\Fil}^n_{P_K} \RR[[x]]=x^n\RR[[x]]=\RR[[x]]_n,\quad  n\geq 0,
\mlabel{eq:opergen}
\end{equation}
by induction on $n$.

For the initial step, for any $f\in \RR[[x]]$, we have
\begin{equation}
P_K\bigg(\bigg(\frac{f(x)}{k(x)}\bigg)'\frac{1}{h(x)}\bigg)=k(x)\int_0^x \bigg(\frac{f(t)}{k(t)}\bigg)'dt=f(x)-\dfrac{f(0)}{k(0)}k(x).
\mlabel{eq:2ftc}
\end{equation}
From this and $x\RR[[x]]=\{f\in \RR[[x]]\,|\,f(0)=0\}$, we obtain that $x\RR[[x]]$ is contained in $P_K(\RR[[x]])$ and hence in $\overline{\Fil}^1_{P_K}\RR[[x]]$. For the opposite inclusion, since $P_K(\RR[[x]])\subseteq x\RR[[x]]$ from Eq.~\meqref{eq:termwiseint}, we have $\overline{\Fil}^1_{P_K}\RR[[x]]=\moverline{AP_K(\RR[[x]])}\subseteq x\RR[[x]]$ as $x\RR[[x]]$ is closed.

For the inductive step, suppose that Eq.~\meqref{eq:opergen} holds for a given $n\geq 1$. Then
$g\in x^{n+1}\RR[[x]]$ implies $(g/k)'h^{-1}\in x^n\RR[[x]]=\overline{\Fil}^n_{P_K} \RR[[x]]$.
Thus by Eq.~(\mref{eq:2ftc}), $g$ is in $P_K(\overline{\Fil}^n_{P_K}\RR[[x]])\subseteq \overline{\Fil}^{n+1}_{P_K} \RR[[x]]$.
Hence, we have the inclusion
$$x^{n+1}\RR[[x]]\subseteq \overline{\Fil}^{n+1}_{P_K} \RR[[x]].$$
For the opposite inclusion, the induction hypothesis and Eq.~\meqref{eq:termwiseint} imply that $$P_K(\overline{\Fil}^n_{P_K}\RR[[x]])=P_K(x^n\RR[[x]])\subseteq x^{n+1}\RR[[x]],$$
and hence $AP_K(\overline{\Fil}^n_{P_K}\RR[[x]])\subseteq  x^{n+1}\RR[[x]]$. Thus $\overline{\Fil}^{n+1}_{P_K} \RR[[x]]=\moverline{AP_K(\overline{\Fil}^n_{P_K}\RR[[x]])}\subseteq  x^{n+1}\RR[[x]]$,
since for the topology given by the filtration $\{x^n\RR[[x]]\}_{n\geq0}$, every $x^n\RR[[x]]$ is closed as the complement of the open set $\bigcup\limits_{0\neq f\in \RR[x],\deg f<n}(f+x^n\RR[[x]])$.
This completes the induction and thereby completes the proof.
\end{proof}

\subsubsection{The free objects} The above preparations give us the suitable category in which to construct the free objects. We give their definitions here and provide a construction in the next subsection.

\begin{defn} \mlabel{df:wra}
Fix a commutative pointed algebra $(A,\lambda)$. The {\bf free \ocmra} over $(A,\lambda)$ is defined to be the initial object in the category $\crac$ of \ocmras. More precisely, it is a \ocmra over $(A,d)$ $$(R(A,\lambda),R(A,\lambda)_n,P_{R(A,\lambda)},\free{\lambda},i_{R(A,\lambda)})$$
that satisfies the desired universal property that,
for every \ocmra $(R,R_n,P_R,\lambda_R,i_R)$ over $(A,\lambda)$ (with the structural pointed algebra homomorphism $(f=)i_R:(A,\lambda)\rar (R,\lambda_R)$),\footnote{In the usual description of the universal property of a free object, $i_R$ plays the role of $f$ from which a morphism $\bar{f}$ on the free object is induced.} there exists a unique homomorphism of \ocmras over $(A,d)$
$$\bar{f}:(R(A,\lambda),R(A,\lambda)_n,P_{R(A,\lambda)},\free{\lambda},i_{R(A,\lambda)})\to (R,R_n,P_R,\lambda_R,i_R).$$
In particular, as a homomorphism over $(A,\lambda)$, we have  $\bar{f}\,i_{R(A,\lambda)}=i_R$. 
\end{defn} 	

We similarly define the notion of the free \ocmdra over a \ddiff algebra. Recall that a \ddiff algebra $(R,D)$ is naturally a pointed algebra $(R,D(1_R))$.

\begin{defn} \mlabel{df:unimod}
Fix a \ddiff algebra $(A,d)$. The {\bf free \ocmdra} over $(A,d)$ is defined to be the initial object in the category $\cdrac$ of \ocmdras over $(A,d)$. More precisely, it is a \ocmdra over $(A,d)$
$$(R(A,d),R(A,d)_n,D_{R(A,d)},P_{R(A,d)},i_{R(A,d)})$$
with the universal property that,
for every \ocmdra $(R,R_n,D_R,P_R,i_R)$ over $(A,d)$ (with the structural \ddiff algebra homomorphism $i_R:(A,d)\rar (R,D_R)$), there exists a unique homomorphism of \ocmdras over $(A,d)$  $$\bar{f}:(R(A,d),R(A,d)_n,D_{R(A,d)},P_{R(A,d)},i_{R(A,d)})\to (R,R_n,D_R,P_R,i_R).$$
In particular, $\bar{f}\,i_{R(A,d)} = i_R$. 
\end{defn} 	

{\em As a shorthand notation, we will denote
	$j_A=i_{R(A,\lambda)}$ in Definitions~\mref{df:wra} and $j_A=i_{R(A,d)}$ in Definition~\mref{df:unimod} in the sequel. }

\subsection{The construction of free commutative \wtd Reynolds algebras and free \ddiffrey algebras}

In this subsection we construct the free objects in the category of commutative \ocmras, and then in the category of commutative \ocmdras.

\noindent
{\bf Notation. } For the rest of this section all algebras are assumed to be commutative.

Fix a base field $\bfk$ and a pointed $\bfk$-algebra $(A,\lambda)$ with $\lambda\in A$. Consider the infinite product
of vector spaces
\begin{equation}
	\shai(A)\coloneqq\prod_{k\geq 1} A^{\ot k}.
	\mlabel{pre}
\end{equation}
Here the tensor powers are taken over $\bk$.

Our goal is to equip $\shai(A)$ with the structure of a free \ocmdra. The statement is given in Theorem~\mref{thm:freedrey} which will be proved in the following steps.

\begin{enumerate}
\item [\bf Step 1.] For the given algebra $A$, equip $\shai(A)$ with the structure of a \compopgen (\S~\mref{sss:csp});
\item [\bf Step 2.] Fix $\lambda\in A$. Use the topology from Step 1 to equip the operated $A$-module $\shai(A)$ with a product $\shpr$, making it into a \wtd Reynolds algebra~(\S~\mref{sss:toa});
\item [\bf Step 3.] Show that the \ocmra\ $(\shai(A),\shai(A)_n,\shpr,P_{\sshai(A)})$, together with the natural inclusion $j_A:A\to \shai(A)$, has the desired universal property as a free \ocmra\ over $(A,\lambda)$~(\S~\mref{sss:fcwra});
\item [\bf Step 4.] Assume that $(A,d)$ is a \ddiff algebra. Further introduce a \ddiff operator $D_{\sshai(A)}$ on $\shai(A)$, and show that $$(\shai(A),\shai(A)_n,\shpr,D_{\sshai(A)},P_{\sshai(A)}, j_A)$$
    is a free \ocmdra\ over $(A,d)$~(\S~\mref{sss:fcdra}).
\end{enumerate}

\subsubsection{{\bf Step 1.} $\shai(A)$ as a \compopgen}
\mlabel{sss:csp}

Let $A$ be a unitary $\bfk$-algebra. We first equip the $\bfk$-module
\begin{equation}
\shai(A)\coloneqq \prod_{k\geq1} A^{\ot k}
\mlabel{eq:shsp}
\end{equation}
with an operated $A$-module structure.
Define the $A$-module action on $\shai(A)$ by multiplying on the left-most tensor factors: for $a\in A$ and a pure tensor $\fraka=a_1\ot a_2\ot \cdots \ot a_k, a_i\in A, 1\leq i\leq k, k\geq 1$, we define
$$ a \fraka := aa_1\ot a_2\ot \cdots \ot a_k.$$
Define a $\bk$-linear operator
\begin{equation}
P_{\sshai(A)}: \shai(A)\to \shai(A), \quad \fraka\mapsto 1_A\ot \fraka, \quad \fraka\in A^{\ot k},\, k\geq 1.
\mlabel{eq:rey-oper}
\end{equation}

Denote
$$\shai(A)_n\coloneqq \prod_{k\geq n+1} A^{\ot k},\quad n\geq0.$$
Then we get a decreasing filtration $\{\shai(A)_n\}_{n\geq0}$ of operated $A$-submodules of $\shai(A)$.

Note that $\shai(A)$ is naturally identified with the inverse limit of $A$-modules:
\[\shai(A)\cong \varprojlim\bigoplus_{i=1}^n A^{\ot i}\cong \varprojlim\,\shai(A)/\shai(A)_n.\]
So $(\shai(A),\shai(A)_n,P_{\sshai(A)})$ is a complete operated  $A$-module.
It can be regarded as the space of $\bk$-linear formal series from tensor powers of $A$, of the form $\sum_{k\geq 0} \fraka_k$ with $\fraka_k\in A^{\ot k}, k\geq 0$.

Also, denote
\[\sha^+(A)\coloneqq \bigoplus_{k\geq0}A^{\ot k}\text{ and }
\shai{^+}(A):=\bfk\oplus \prod_{k\geq 1} A^{\ot k},\]
again identified with the inverse limit
$$
\shai{^+}(A)\cong \varprojlim\bigoplus_{i=0}^k A^{\ot i} \cong \varprojlim\left(\bigoplus_{i\geq0} A^{\ot i}\bigg/\bigoplus_{i>k} A^{\ot i}\right),$$
with the convention that $A^{\otimes 0}\coloneqq \bk$.
Then as a complete vector space, we have the identification
\begin{equation}
	\shai(A)=A\,\hat{\ot}\, \shai{^+}(A)\cong \varprojlim\left(\Big(A\ot \bigoplus_{i\geq0} A^{\ot i}\Big)\bigg/\Big(A\ot \bigoplus_{i>k} A^{\ot i}\Big)\right).
\mlabel{eq:auga}
\end{equation}
Note that here $\hat{\ot}$ is the complete tensor product with respect to the inverse limit topology on $\shai{^+}(A)$. So a typical element of $A\hat{\ot}\, \shai{^+}(A)$ is of the form $\sum\limits_{k\geq 0} \sum\limits_i a_{k,i}\ot \fraka_{k,i}$, where $\sum\limits_i$ is a finite sum with $a_{k,i}\in A$ and $\fraka_{k,i}\in A^{\ot k}, k\geq 0$.

\begin{prop}
\mlabel{prop:ctoa}
The triple $(\shai(A),\shai(A)_n,P_{\sshai(A)})$ is a \compopgen. Namely, we have
\begin{equation}
\mlabel{eq:compfiln}
\shai(A)_n=\overline{\Fil}^n_{P_{\sshai(A)}}\shai(A),\quad n\geq0.
\end{equation}
\end{prop}

\begin{proof}
For clarity, in the proof we will use the abbreviations $\shai=\shai(A)$, $\shai_n=\shai(A)_n,\ n\geq 0,$ and $P=P_{\sshai(A)}$.

By the definition of $(\shai,P)$, we have
$P(\shai_n)\subseteq \shai_{n+1},\ n\geq0$. So $\{\shai_n\}_{n\geq0}$ is a decreasing filtration of operated $A$-submodules of $\shai$.

The inclusions $\overline{\Fil}^n_P \shai \subseteq \shai_n,\ n\geq 0,$ follow from starting with
$\overline{\Fil}^0_P \shai=\shai_0=\shai$ and then applying the induction on $n$ to yield
$$\overline{\Fil}^n_P \shai=\moverline{AP(\overline{\Fil}^{n-1}_P \shai)}\subseteq \moverline{AP(\shai_{n-1})}\subseteq\shai_n,\quad n\geq1.$$
Here the last inclusion follows from $AP(\shai_{n-1})\subseteq\shai_n$ and the fact that $\shai_n$ is closed since it is the complement of the open set $\bigcup\limits_{0\neq f\in \bigoplus_{i=1}^n A^{\ot i}}(f+\shai_n).$

We next prove the opposite inclusion $\shai_n \subseteq \overline{\Fil}^n_P \shai$ by contradiction. Suppose that there is $f\in \shai_n \backslash \overline{\Fil}^n_P \shai=\shai_n \cap \Big(\overline{\Fil}^n_P \shai\Big)^c$. Write $f=\sum\limits_{i= n+1}^\infty f_i$ with $f_i\in A^{\ot i}$.
Since $\Big(\overline{\Fil}^n_P \shai\Big)^c$ is open, there exists an open neighborhood $f+\shai_m$ of $f$ with $m>n$ that  is contained in $\Big(\overline{\Fil}^n_P \shai\Big)^c$. Therefore,
$$\sum_{i=n+1}^m f_i=f-\sum_{i=m+1}^\infty f_i\in f+\shai_m\subseteq (\overline{\Fil}^n_P \shai)^c.$$
However, for $\fraka=a_0\ot a_1\ot \cdots \ot a_{i-1}\in A^{\ot i}$ with $i\geq n+1$, we have
$$\fraka=a_0 P(a_1 P(\cdots P(a_{i-1})\cdots)\in \overline{\Fil}^n_P \shai.$$
Thus $\sum\limits_{i=n+1}^m f_i$ is in $\overline{\Fil}^n_P \shai$. This is a contradiction.
In summary, we have obtained that $\shai_n=\overline{\Fil}^n_P \shai,\ n\geq0$. This completes the proof.
\end{proof}

\subsubsection{{\bf Step 2.} $\shai(A)$ as a \wtd Reynolds algebra}
\mlabel{sss:toa}
For the algebraic product on $\shai(A)$, we will use a complete version of the shuffle product.
Recall the usual shuffle product $\ssha$ on the tensor space $T(A)$ (here denoted $\sha^+(A)$), recursively defined by
\begin{equation}\mlabel{eq:sh}
1_\bk\ssha1_\bk=1_\bk,\,\fraka\,\ssha 1_\bk=1_\bk\ssha\,\fraka=\fraka,\,\fraka\,\ssha\,\frakb=a_1\ot(\fraka'\,\ssha\,\frakb)+ b_1\ot(\fraka\,\ssha\,\frakb'),
\end{equation}
for $\fraka=a_1\ot\fraka'\in A^{\ot m}$,  $\frakb=b_1\ot\frakb'\in A^{\ot n}$ with $m,n>0$. To make the product explicit, denote
$$ \frakc\coloneqq c_1\ot \cdots \ot c_{m+n}\coloneqq \fraka\ot \frakb= a_1\ot \cdots \ot a_m \ot b_1\ot \cdots \ot b_n,
$$
and let

\[S_{m,n}\coloneqq \left\{\sigma\in S_{m+n}\,\big|\,\sigma^{-1}(1)<\cdots<\sigma^{-1}(m),\,\sigma^{-1}(m+1)<\cdots<\sigma^{-1}(m+n)\right\}\]
be the set of $(m,n)$-unshuffles. Then $c_{\sigma(1)}\ot \cdots \ot c_{\sigma(m+n)}$ are called the {\bf shuffles} of $\fraka$ and $\frakb$ and the shuffle product of $\fraka$ and $\frakb$ is

\begin{equation}
\fraka \,\ssha\, \frakb=\sum_{\sigma\in S_{m,n}} c_{\sigma(1)}\ot \cdots \ot c_{\sigma(m+n)}.
\mlabel{eq:shuf}
\end{equation}

We give another description of the shuffle product for later applications. Partition $S_{m,n}$ into the disjoint subsets
$$ S'\coloneqq\{\sigma\in S_{m,n}\,|\,\sigma^{-1}(m)=m+n\}, \quad
S''\coloneqq\{\sigma\in S_{m,n}\,|\, \sigma^{-1}(m+n)=m+n\}. $$
So $S'$ (resp. $S''$) consists of unshuffles $\sigma$ for which $c_{\sigma(m+n)}=a_m$ (resp. $c_{\sigma(m+n)}=b_n$) in $c_{\sigma(1)}\ot \cdots \ot c_{\sigma(m+n)}$. For $\sigma\in S'$, let $k=\sigma^{-1}(m+n)$, that is, $c_{\sigma(k)}=b_n$. Likewise, for $\sigma\in S''$, let $\ell=\sigma^{-1}(m)$, so that $c_{\sigma(\ell)}=a_m$.
Then the shuffle product in Eq.~\meqref{eq:shuf} can be rewritten as
\begin{equation}
\fraka \,\ssha\, \frakb=\sum_{\sigma\in S'}
c_{\sigma(1)}\ot \cdots \ot \underbrace{c_{\sigma(k)}}_{=b_n}\ot a_{k-n+1}\ot \cdots \ot a_m
+\sum_{\sigma\in S''}
c_{\sigma(1)}\ot \cdots \ot \underbrace{c_{\sigma(\ell)}}_{=a_m}\ot b_{\ell-m+1}\ot \cdots \ot b_n.
\mlabel{eq:shuf2}
\end{equation}

Now fix a pointed algebra $(A,\lambda)$.
With the above notations, we define a binary operation, called the  {\bf complete shuffle product}, 
 $$\hat{\ssha}:\shai{^+}(A)\,\underline{\ot}\,\shai{^+}(A)\rar\shai{^+}(A).
$$
Here the notion of underlined tensor $\underline{\ot}$ is used to distinguish it from the tensor symbol in $\shai{^+}(A)$.
For $\fraka=a_1\ot\cdots\ot a_m\in A^{\ot m}$ and $\frakb=b_1\ot\cdots\ot b_n\in A^{\ot n}$ with $m,n\geq 1$, define
\begin{align}
&
1_\bk\hat{\ssha}1_\bk\coloneqq 1_\bk,\quad\fraka\,\hat{\ssha}1_\bk=1_\bk\hat{\ssha}\,\fraka\coloneqq \fraka,\mlabel{eq:astexp1} \\
&\fraka\,\hat{\ssha}\,\frakb\coloneqq \sum_{\sigma\in S_{m,n}}\sum_{\fraki\in I_\sigma}(-\lambda)^{\ot i_1}\ot c_{\sigma(1)}\ot (-\lambda)^{\ot i_2}\ot c_{\sigma(2)}\ot \cdots\ot (-\lambda)^{\ot i_{m+n}} \ot c_{\sigma(m+n)}, \quad m, n>0,
\mlabel{eq:astexp2}
\end{align}
where, for $\sigma\in S_{m,n}$, the set $I_\sigma$ consists of all tuples $\fraki=(i_1,\dots,i_{m+n})\in\mathbb N^{m+n}$ such that $i_j=0$ whenever $j>\min\{\sigma^{-1}(m), \sigma^{-1}(m+n)\}$.

Alternatively, for a shuffle $c_{\sigma(1)}\ot \cdots \ot c_{\sigma(m+n)}$ of tensors $\fraka$ and $\frakb$ with $m, n\geq 1$, an {\bf extension} of this shuffle is obtained by inserting an arbitrary tensor power $(-\lambda)^{\ot j}$ in front of each of the tensor factors $c_{\sigma(j)}$ until $c_{\sigma(j)}$ is either $a_m$ or $b_n$, after which no insertion is allowed. Then $\fraka \,\hat{\ssha}\, \frakb$ is the sum of all the extensions of all the shuffles.
More precisely, applying the notions in Eq.~\meqref{eq:shuf2},
 Eq.~\meqref{eq:astexp2} can be rewritten as
{\small
\begin{equation}
\begin{split}
\fraka\,\hat{\ssha}\,\frakb=&
\sum_{\sigma\in S'}\sum_{i_j\geq 0,\ 1\leq j\leq k}(-\lambda)^{\ot i_1}\ot c_{\sigma(1)}\ot (-\lambda)^{\ot i_2}\ot c_{\sigma(2)}\ot \cdots\ot (-\lambda)^{\ot i_k} \ot c_{\sigma(k)}\ot a_{k-n+1}\ot \cdots \ot a_m\\
&+\sum_{\sigma\in S''}\sum_{i_j\geq 0,\ 1\leq j\leq \ell}(-\lambda)^{\ot i_1}\ot c_{\sigma(1)}\ot (-\lambda)^{\ot i_2}\ot c_{\sigma(2)}\ot \cdots\ot (-\lambda)^{\ot i_\ell} \ot c_{\sigma(\ell)}\ot b_{\ell-m+1}\ot \cdots \ot b_n.
\end{split}
\end{equation}
}

By Eq.~\meqref{eq:astexp2}, for each $r\geq 1$, the number of pure tensors in $\fraka\,\hat{\ssha}\,\frakb$ of length $r$ is $0$ for $r<m+n$. For $r\geq m+n$, this number is no more than the number of tuples $(i_1,\ldots,i_{m+n})\in \NN^{m+n}$ with $i_1+\cdots+i_{m+n}=r-m-n$, that is, the number of weak compositions of length $m+n$ and weight (sum) $r-m-n$. Hence this number is finite.
Thus $\fraka\,\hat{\ssha}\,\frakb$ is a well-defined element of $\prod_{r\geq m+n} A^{\ot r}$.
Then we can linearly extend $\hat{\ssha}$ to the general case when $\fraka$ and $\frakb$ are formal series of pure tensors and obtain a well-defined binary operation $\hat{\ssha}$ on $\shai{^+}(A)$.

As illustrations of Eq.~\eqref{eq:astexp2}, we give some examples for small values of $m$ and $n$.
When $m=n=1$, then from $a\ssha b= a\ot b+b\ot a$, we have
\[a\,\hat{\ssha}\,b
=\sum_{k\geq0} (-\lambda)^{\ot  k}\otimes(a\ot b+b\ot a), \quad a, b\in A.
\]  	
When $m=n=2$, then $\fraka=a_1\ot a_2,\,\frakb=b_1\ot b_2\in A^{\ot 2}$. Then from the shuffle product
\begin{align*}
\fraka\ssha \frakb=&
a_1\ot a_2\ot b_1\ot b_2+b_1\ot b_2\ot a_1\ot a_2
+a_1\ot b_1\ot a_2\ot b_2\\
&+a_1\ot b_1\ot b_2\ot a_2
+b_1\ot a_1\ot a_2\ot b_2+b_1\ot a_1\ot b_2\ot a_2,
\end{align*}j
we obtain

\begin{align*}
\fraka\,\hat{\ssha}\,\frakb
&=\sum_{i_1,i_2\geq0}(-\lambda)^{\ot i_1}\ot a_1 \ot (-\lambda)^{\ot i_2}\ot a_2\ot b_1\ot b_2\\
&\qquad
 +(-\lambda)^{\ot i_1}\ot b_1 \ot (-\lambda)^{\ot i_2}\ot b_2\ot a_1\ot a_2\\
&+\sum_{i_1,i_2,i_3\geq0}(-\lambda)^{\otimes i_1}\ot a_1 \ot (-\lambda)^{\otimes i_2}\ot b_1 \ot (-\lambda)^{\otimes i_3}\ot a_2\ot b_2\\
&\qquad
+(-\lambda)^{\otimes i_1}\ot a_1 \ot (-\lambda)^{\otimes i_2}\ot b_1 \ot (-\lambda)^{\otimes i_3}\ot b_2\ot a_2\\
&\qquad
+(-\lambda)^{\otimes i_1}\ot b_1 \ot (-\lambda)^{\otimes i_2}\ot a_1 \ot (-\lambda)^{\otimes i_3}\ot a_2\ot b_2\\
&\qquad
+(-\lambda)^{\otimes i_1}\ot b_1 \ot (-\lambda)^{\otimes i_2}\ot a_1 \ot (-\lambda)^{\otimes i_3}\ot b_2\ot a_2.
\end{align*}

We now give a recursive characterization of $\hat{\ssha}$.
\begin{prop}
The binary operation
$\hat{\ssha}$ on $\shai{^+}(A)$ defined in Eqs.~\eqref{eq:astexp1} and \eqref{eq:astexp2}
has the following recursion. For $\fraka=a_1\ot\fraka'\in A^{\ot m}$ and $\frakb=b_1\ot\frakb'\in A^{\ot n},\,m,n\geq 1$, with the convention that $\fraka'=1_\bk$ $($resp. $\frakb'=1_\bk$$)$ when $m=1$ $($resp. $n=1$$)$, we have
\begin{equation}
\mlabel{eq:ast2}
\fraka\,\hat{\ssha}\,\frakb=a_1\ot(\fraka'\,\hat{\ssha}\,\frakb)+ b_1\ot(\fraka\,\hat{\ssha}\,\frakb')-\lambda\ot(\fraka\,\hat{\ssha}\,\frakb).
\end{equation}
\end{prop}
Note that, even though the left-hand side $\fraka\,\hat{\ssha}\,\frakb$ also appears on the right, the latter  has a higher tensor order, giving rise to a recursion in the complete space.
\begin{proof}
When $m=n=1$, we have $\fraka=a$ and $\frakb=b$ with $a, b\in A,$ and the recursion in \eqref{eq:ast2} is verified by 
\begin{align*}
a\,\hat{\ssha}\,b&=\sum_{k\geq0} (-\lambda)^{\ot  k}\otimes(a\ot b+b\ot a)\\
&=a\ot b+b\ot a+(-\lambda)\ot\Big(\sum_{k\geq0} (-\lambda)^{\ot  k}\otimes(a\ot b+b\ot a)\Big)\\
&=a\ot (1_\bk\hat{\ssha}\,b)+b\ot (a\,\hat{\ssha}1_\bk)-\lambda\ot(a\,\hat{\ssha}\,b).
\end{align*}  	
One can similarly check the case when exactly one of $m$ or $n$ is 1.
Now for $\fraka=a_1\ot\fraka'\in A^{\ot m}$ and $\frakb=b_1\ot\frakb'\in A^{\ot n}$ with $m,n>1$, we first divide the right-hand side of \eqref{eq:astexp2} into two summands:

$$\fraka\,\hat{\ssha}\,\frakb=T_1+T_2,$$
where
\[\begin{array}{l}
T_1:=\sum\limits_{\sigma\in S_{m,n}}\sum\limits_{\fraki\in I_\sigma\atop i_1=0}(-\lambda)^{\ot i_1}\ot c_{\sigma(1)}\ot \cdots\ot (-\lambda)^{\ot i_{m+n}} \ot c_{\sigma(m+n)},\\
T_2:=\sum\limits_{\sigma\in S_{m,n}}\sum\limits_{\fraki\in I_\sigma\atop i_1>0}(-\lambda)^{\ot i_1}\ot c_{\sigma(1)}\ot \cdots\ot (-\lambda)^{\ot i_{m+n}} \ot c_{\sigma(m+n)}.
\end{array}\]
Writing $\fraka'\ot\frakb=d_1\ot\dots\ot d_{m+n-1}$ and $\fraka\ot\frakb'=e_1\ot\dots\ot e_{m+n-1}$, for the first summand $T_1$, we have
\begin{align*}
T_1=&
\sum_{\sigma\in S_{m,n}\atop \sigma(1)=1}\sum_{\fraki\in I_\sigma\atop i_1=0}(-\lambda)^{\ot i_1}\ot c_{\sigma(1)}\ot \cdots\ot (-\lambda)^{\ot i_{m+n}} \ot c_{\sigma(m+n)}\\
&+\sum_{\sigma\in S_{m,n}\atop \sigma(1)=m+1}\sum_{\fraki\in I_\sigma\atop i_1=0}(-\lambda)^{\ot i_1}\ot c_{\sigma(1)}\ot \cdots\ot (-\lambda)^{\ot i_{m+n}} \ot c_{\sigma(m+n)}\\
\stackrel{\eqref{eq:sh}}{=}&a_1\ot\Big(\sum_{\tau\in S_{m-1,n}}\sum_{(j_1,\dots,j_{m+n-1})\in I_\tau}(-\lambda)^{\ot j_1}\ot d_1\ot \cdots\ot (-\lambda)^{\ot j_{m+n-1}}\ot d_{m+n-1}\Big)\\
&+ b_1\ot\Big(\sum_{\omega\in S_{m,n-1}}\sum_{(k_1,\dots,k_{m+n-1})\in I_\omega}(-\lambda)^{\ot k_1}\ot e_1\ot \cdots\ot (-\lambda)^{\ot k_{m+n-1}} \ot e_{m+n-1}\Big)\\
=&a_1\ot(\fraka'\,\hat{\ssha}\,\frakb)+b_1\ot(\fraka\,\hat{\ssha}\,\frakb').
\end{align*}
On the other hand,
$$
T_2=(-\lambda)\ot\Big(\sum\limits_{\sigma\in S_{m,n}}\sum\limits_{\fraki\in I_\sigma\atop i_1>0}(-\lambda)^{\ot i_1-1}\ot c_{\sigma(1)}\ot (-\lambda)^{\ot i_2}\ot c_{\sigma(2)}\ot\cdots\ot (-\lambda)^{\ot i_{m+n}} \ot c_{\sigma(m+n)}\Big)
\stackrel{\eqref{eq:astexp2}}{=}-\lambda\ot(\fraka\,\hat{\ssha}\,\frakb).
$$
Hence, we have obtained the desired equality \meqref{eq:ast2}.
\end{proof}

We next define a binary operation on $\shai(A)$ by an extension of scalar from $\hat{\ssha}$.
\begin{equation}\label{eq:shpr}
\begin{split}
\shpr:&\shai(A)\,\underline{\ot} \,\shai(A)\rar\shai(A), \\
& \fraka\shpr\frakb\coloneqq a_0b_0\ot(\fraka'\,\hat{\ssha}\,\frakb'),\
\fraka=a_0\ot\fraka'\in A^{\ot m}, \frakb=b_0\ot\frakb'\in A^{\ot n}, \ m,n\geq1.
\end{split}
\end{equation}
with the convention in Eq.~\meqref{eq:auga}.

\begin{prop}\label{prop:ccdrey}
Given a pointed algebra $(A,\lambda)$,  	
the quadruple $\big(\shai(A),\shpr,P_{\sshai(A)},\lambda\big)$ is an $(A,\lambda)$-\wtd Reynolds algebra.
\end{prop}
\begin{proof}
First we show that $\hat{\ssha}$ is associative, and hence so is $\shpr$ defined in Eq.~\meqref{eq:shpr}.		

If at least one of $\fraka,\frakb$ or $\frakc$ is $1_\bk$, then it is clear that $(\fraka\,\hat{\ssha}\,\frakb)\,\hat{\ssha}\,\frakc=\fraka\,\hat{\ssha}\,(\frakb\,\hat{\ssha}\,\frakc)$. In the remaining case, we write $\fraka=a_1\ot\fraka'\in A^{\ot m},\,\frakb=b_1\ot\frakb'\in A^{\ot n}$ and $\frakc=c_1\ot\frakc'\in A^{\ot l},\,m,n,l>0$. 
 The associativity of $\hat{\ssha}$ can be checked by induction on $m+n+l\geq 3$ as follows.

First applying the equality
\eqref{eq:ast2} to $(\fraka\,\hat{\ssha}\,\frakb)\,\hat{\ssha}\,\frakc$ and $\fraka\,\hat{\ssha}\,(\frakb\,\hat{\ssha}\,\frakc)$ twice respectively, we obtain
\begin{align*}
(\fraka\,\hat{\ssha}\,\frakb)\,\hat{\ssha}\,\frakc&
=a_1\ot\big((\fraka'\,\hat{\ssha}\,\frakb)\,\hat{\ssha}\,\frakc\big)+ b_1\ot\big((\fraka\,\hat{\ssha}\,\frakb')\,\hat{\ssha}\,\frakc)-\lambda\ot\big((\fraka\,\hat{\ssha}\,\frakb)\,\hat{\ssha}\,\frakc\big)\\
&\quad+c_1\ot\big((\fraka\,\hat{\ssha}\,\frakb)\,\hat{\ssha}\,\frakc'\big)-\lambda\ot\big((\fraka\,\hat{\ssha}\,\frakb)\,\hat{\ssha}\,\frakc\big),\\[.5em]
\fraka\,\hat{\ssha}\,(\frakb\,\hat{\ssha}\,\frakc)&
=a_1\ot(\fraka'\,\hat{\ssha}\,(\frakb\,\hat{\ssha}\,\frakc))\\ &\quad+(b_1\ot(\fraka\,\hat{\ssha}\,(\frakb'\,\hat{\ssha}\,\frakc))+c_1\ot(\fraka\,\hat{\ssha}\,(\frakb\,\hat{\ssha}\,\frakc'))
-\lambda\ot(\fraka\,\hat{\ssha}\,(\frakb\,\hat{\ssha}\,\frakc)))\\
&\quad-\lambda\ot(\fraka\,\hat{\ssha}\,(\frakb\,\hat{\ssha}\,\frakc)).
\end{align*}
In terms of the linear operator
$$Q_\lambda: \shai{^+}(A)\to \shai{^+}(A), \quad \fraka\mapsto \fraka+2\lambda\ot\fraka,$$
above equalities are equivalent to
\begin{align*}
	&Q_\lambda((\fraka\,\hat{\ssha}\,\frakb)\,\hat{\ssha}\,\frakc)=a_1\ot((\fraka'\,\hat{\ssha}\,\frakb)\,\hat{\ssha}\,\frakc)+ b_1\ot((\fraka\,\hat{\ssha}\,\frakb')\,\hat{\ssha}\,\frakc)+c_1\ot((\fraka\,\hat{\ssha}\,\frakb)\,\hat{\ssha}\,\frakc'),\\
	&Q_\lambda(\fraka\,\hat{\ssha}\,(\frakb\,\hat{\ssha}\,\frakc))=a_1\ot(\fraka'\,\hat{\ssha}\,(\frakb\,\hat{\ssha}\,\frakc))
	+b_1\ot(\fraka\,\hat{\ssha}\,(\frakb'\,\hat{\ssha}\,\frakc))+c_1\ot(\fraka\,\hat{\ssha}\,(\frakb\,\hat{\ssha}\,\frakc')).
\end{align*}
Now applying the induction hypothesis to the second tensor factor of each term on the right hand sides, we obtain 
 $$Q_\lambda((\fraka\,\hat{\ssha}\,\frakb)\,\hat{\ssha}\,\frakc)=Q_\lambda(\fraka\,\hat{\ssha}\,(\frakb\,\hat{\ssha}\,\frakc)).$$
Note that the operator $Q_\lambda$ has the inverse operator
$$Q_\lambda^{-1}:\shai{^+}(A) \to \shai{^+}(A), \quad \fraka\mapsto \fraka+\sum_{r\geq1}(-2\lambda)^{\ot r}\ot\fraka$$
and hence is injective.
Thus the associativity of $\hat{\ssha}$ is verified.

Second, we show that $P_{\sshai(A)}$ is a \wtd Reynolds operator on $\shai(A)$ satisfying condition \meqref{Gra}.
Indeed, for $\fraka=a_0\otimes\fraka'\in A^{\ot m}$ and $\frakb=b_0\ot\frakb'\in A^{\ot n},\,m,n\geq1$, we have
\begin{align*}
P_{\sshai(A)}(\fraka)\shpr P_{\sshai(A)}(\frakb)&\stackrel{\eqref{eq:shpr},\,\eqref{eq:rey-oper}}{=}1_A\ot(\fraka\,\hat{\ssha}\,\frakb)\\
&\stackrel{\eqref{eq:ast2}}{=}
1_A\ot(a_0\ot(\fraka'\,\hat{\ssha}\,\frakb)+b_0\ot(\fraka\,\hat{\ssha}\,\frakb')-\lambda\ot(\fraka\,\hat{\ssha}\,\frakb))\\
&\stackrel{\eqref{eq:shpr},\,\eqref{eq:rey-oper}}{=}P_{\sshai(A)}(\fraka\shpr P_{\sshai(A)}(\frakb)+P_{\sshai(A)}( \fraka)\shpr\frakb-\lambda\shpr P_{\sshai(A)}(\fraka)\shpr P_{\sshai(A)}(\frakb)).
\end{align*}

In summary, we have shown that $\big(\shai(A),\shpr,P_{\sshai(A)},\lambda\big)$ is a \wtd Reynolds algebra.
\end{proof}

\subsubsection{{\bf Step 3.} $\shai(A)$ as the free \ocmra}
\mlabel{sss:fcwra}

Let $j_A:A\rar \shai(A),\,a\mapsto a$ be the natural injection.

\begin{theorem}
Given a pointed algebra $(A,\lambda)$,  	
the sextuple $\big(\shai(A),\shai(A)_n,\shpr,P_{\sshai(A)},\lambda,j_A\big)$ is the free object in the category $\crac$ of \ocmra\ over $(A,\lambda)$.
\mlabel{th:fcdr}
\end{theorem}

\begin{proof}
By Proposition~\ref{prop:ccdrey}, it remains to verify the universal property of $\shai(A)$ as the free \ocmra over $(A,\lambda)$.

Let $(R,R_n,P_R,\lambda_R,i_R)\in\crac$ be a \ocmra over $(A,\lambda)$ (with a pointed algebra homomorphism $i_R:(A,\lambda)\rar (R,\lambda_R)$ as the structure map), we will construct an $A$-module map $\bar{f}:\shai(A)\rar R$ as follows. First define the $A$-module map
\[\tilde{f}:\bigoplus_{i\geq1} A^{\ot i}\rar R,\, A^{\ot(n+1)}\ni \fraka=a_0\ot\cdots\ot a_n\mapsto i_R(a_0)P_R\big(i_R(a_1)P_R(\cdots P_R(i_R(a_n))\cdots)\big), \quad n\geq 0.\]
Then $\tilde{f}(A^{\ot (n+1)})\subseteq AP_R\big(\tilde{f}(A^{\ot n})\big), n\geq 1$.
Thus from $\tilde{f}(A)=i_R(A)\subseteq R_0=R$, we recursively obtain
$$ \tilde{f}(A^{\ot (n+1)})\subseteq AP_R\big(\tilde{f}(A^{\ot n})\big) \subseteq
	\moverline{ A P(\overline{\Fil}^{n-1}_PR)} =
\overline{\Fil}^n _PR=R_n, \quad n\geq 1,$$
and hence
$$\tilde{f}\,\bigg(\bigoplus_{i\geq n+1} A^{\ot i}\bigg)
\subseteq\overline{\Fil}^n _PR=R_n, \quad n\geq 0.$$

Correspondingly, we have an induced  $A$-module map
$$\prod_{n\geq0}\bigg(\bigoplus_{i\geq1} A^{\ot i}\bigg/\bigoplus_{i\geq n+1} A^{\ot i}\bigg)\to \prod_{n\geq0}R/R_n,\ \Big(\fraka_n+\bigoplus_{i\geq n+1} A^{\ot i}\Big)_{n\geq0}\mapsto \big(\tilde{f}(\fraka_n)+R_n\big)_{n\geq0}.$$
If $(\fraka_n+\bigoplus_{i\geq n+1} A^{\ot i})_{n\geq0}\in \varprojlim\left(\bigoplus_{i\geq1} A^{\ot i}\bigg/\bigoplus_{i\geq n+1} A^{\ot i}\right)$, then
$\fraka_{n+1}-\fraka_n\in \bigoplus_{i\geq n+1} A^{\ot i}$ implies that $\tilde{f}(\fraka_{n+1})-\tilde{f}(\fraka_n)=\tilde{f}(\fraka_{n+1}-\fraka_n)\in R_n$ for all $n\geq0$,
so we obtain a $A$-module map
$$\hat{f}:\shai(A)=\varprojlim\left(\bigoplus_{i\geq1} A^{\ot i}\bigg/\bigoplus_{i\geq n+1} A^{\ot i}\right)\rar  \hat{R}= \varprojlim R/R_n,\  (\fraka_n+\bigoplus_{i\geq n+1} A^{\ot i})_{n\geq0}\mapsto (\tilde{f}(\fraka_n)+R_n)_{n\geq0}.$$

Let $\bar{f}=\kappa_R^{-1}\hat{f}$. We show that it is our desired $A$-module map from $\shai(A)$ to $R$ such that  $i_R=\bar{f}\circ j_A$.

Since $\tilde{f}\,\bigg(\bigoplus_{i\geq n+1} A^{\ot i}\bigg)
\subseteq R_n$ and $\shai(A)_n=\varprojlim\left(\left(\bigoplus_{i\geq n+1} A^{\ot i}+\bigoplus_{i\geq k+1} A^{\ot i}\right)\bigg/\bigoplus_{i\geq k+1} A^{\ot i}\right)$, we have
$$\hat{f}(\shai(A)_n)\subseteq \hat{R}_n=\bigg\{(r_k+R_k)_{k\geq0}\in \prod_{k\geq0}(R_n+R_k)/R_k\,|\,r_{k+1}-r_k\in R_k\bigg\},\quad n\geq0.$$

On the other hand, we show that $\kappa_R(R_n)=\hat{R}_n$ below, implying $\bar{f}\,\bigg(\shai(A)_n\bigg)\subseteq R_n,\ n\geq0$.
In fact, by the definition of $\kappa_R$, it is clear that $\kappa_R(R_n)\subseteq\hat{R}_n$.
Conversely, for any $x\in\kappa_R^{-1}(\hat{R}_n)$, we have
$$\kappa_R(x)=(x+R_k)_{k\geq0}\in \hat{R}_n.$$
When $k\geq n$, we see that $x+R_k\in R_n/R_k$, i.e. $x\in R_n$. Hence, $\kappa_R^{-1}(\hat{R}_n)\subseteq R_n$. We have shown that $\kappa_R(R_n)=\hat{R}_n$.

Fix $\fraka=a_0\otimes\fraka'\in A^{\ot m}$ and $\frakb=b_0\ot\frakb'\in A^{\ot n},\,m,n\geq1$. By the construction of $\bar{f}$, we have
\[\bar{f}(P_{\sshai(A)}(\fraka))=\bar{f}(1_A\ot\fraka)=i_R(1_A)P_R(\bar{f}(\fraka))=P_R(\bar{f}(\fraka)).\]
Hence, $\bar{f}\, P_{\sshai(A)}=P_R\,\bar{f}$.
Thus we have shown that $\bar{f}$ is a homomorphism of operated $A$-modules that preserves the filtrations.

Next we check that $\bar{f}$ is a \wtd Reynolds algebra homomorphism. We begin with verifying
 $\bar{f}(\fraka\shpr\frakb)=\bar{f}(\fraka)\bar{f}(\frakb)$ by induction on $m+n$.
When one of $m$ or $n$ is $1$, it holds simply because $i_R$ is an algebra homomorphism. When $m$ and $n$ are both greater than $1$, writing $\fraka=a_1\ot \fraka'$ and $\frakb=b_1\ot \frakb'$, we first prove
\begin{equation}
P_R(\bar{f}(\fraka'\,\hat{\ssha}\,\frakb'))=P_R(\bar{f}(\fraka'))P_R(\bar{f}(\frakb')).
\mlabel{eq:reyhom}
\end{equation}
We rewrite Eq. \meqref{eq:ast2} via \meqref{eq:rey-oper} and \meqref{eq:shpr} as
\[\fraka'\,\hat{\ssha}\,\frakb'=\fraka'\shpr P_{\sshai(A)}(\frakb')+P_{\sshai(A)}( \fraka')\shpr\frakb'-\lambda\shpr P_{\sshai(A)}(\fraka'\,\hat{\ssha}\,\frakb').\]
Applying $\bar{f}$ and then $P_R$ to both sides of this equality,
we obtain
\begin{equation*}
\begin{split}
P_R(\bar{f}(\fraka'\,\hat{\ssha}\,\frakb'))
&=P_R\left(\bar{f}(\fraka')\bar{f}\big(P_{\sshai(A)}(\frakb')\big)+\bar{f}\big(P_{\sshai(A)}( \fraka')\big)\bar{f}(\frakb')-\bar{f}(\lambda)\bar{f}\big(P_{\sshai(A)}(\fraka'\,\hat{\ssha}\,\frakb')\big)\right)\\
&=P_R(\bar{f}(\fraka')P_R(\bar{f}(\frakb')))+P_R(P_R(\bar{f}(\fraka'))\bar{f}(\frakb'))-P_R(\lambda_RP_R(\bar{f}(\fraka'\,\hat{\ssha}\,\frakb'))),
\end{split}
\end{equation*}
by the induction hypothesis and the equality $\bar{f}\, P_{\sshai(A)}=P_R\, \bar{f}$ just obtained above. Equivalently, $P_R(\bar{f}(\fraka'\,\hat{\ssha}\,\frakb'))$ satisfies the  equation
\begin{equation*} 
(\id_R+P_R\lambda_R)\Big(P_R(\bar{f}(\fraka'\,\hat{\ssha}\,\frakb'))\Big)=P_R(\bar{f}(\fraka')P_R(\bar{f}(\frakb')))+P_R(P_R(\bar{f}(\fraka'))\bar{f}(\frakb')).
\end{equation*}
On the other hand, the weighted Reynolds identity \meqref{Gra} gives
$$(\id_R+P_R\lambda_R)\Big(P_R(\bar{f}(\fraka'))P_R(\bar{f}(\frakb'))\Big)=P_R(\bar{f}(\fraka')P_R(\bar{f}(\frakb')))+P_R(P_R(\bar{f}(\fraka'))\bar{f}(\frakb')).
$$
Thus to prove Eq.~\meqref{eq:reyhom}, we just need to prove that the operator $F\coloneqq \id_R+P_R\lambda_R$ on $R$ is invertible. For this, we directly construct its inverse $G$.
In fact, using the notation given after Definition~\mref{def:com}, we consider
the maps (retaining the composition symbol $\circ$ for precision)
$$\kappa_n\circ\Big(\sum\limits_{k=0}^{n-1}(-P_R\lambda_R)^k\Big):R\to R/R_n,\quad n\geq1, $$
and the universal property of $\hat{R}$ as an inverse limit provides a linear map $G':R\rar\hat{R}$ such that $\hat{\kappa}_n\circ G'=\kappa_n\circ\Big(\sum\limits_{k=0}^{n-1}(-P_R\lambda_R)^k\Big)$.
Since $R$ is complete and $\kappa_R:R\to \hat{R}$ is an isomorphism, we take $G\coloneqq \kappa_R^{-1}\circ G'$.
For $n\geq 1$, we have
\begin{align*}
\kappa_n&\circ(G\circ F)=(\kappa_n\circ\kappa_R^{-1})\circ G'\circ F=(\hat{\kappa}_n\circ G')\circ F\\
&=\kappa_n\circ\left(\sum_{k=0}^{n-1}(-P_R\lambda_R)^k\right)\circ(\id_R+ P_R\lambda_R)=\kappa_n\circ(\id_R-(-P_R\lambda_R)^n)=\kappa_n,\\
\kappa_n&\circ(F\circ G)=\kappa_n\circ F\circ \kappa_R^{-1}\circ G'=(\id_{R/R_n}+(P_R)_n\times (\lambda_R+R_n))\circ(\kappa_n\circ\kappa_R^{-1})\circ G'\\
&=(\id_{R/R_n}+(P_R)_n\times (\lambda_R+R_n))\circ(\hat{\kappa}_n\circ G')\\
&=\kappa_n\circ(\id_R+P_R\lambda_R)\circ\left(\sum_{k=0}^{n-1}(- P_R\lambda_R )^k\right)=\kappa_n\circ(\id_R-(-P_R\lambda_R )^n)=\kappa_n.
\end{align*}
Hence, $F\circ G=G\circ F=\id_R$, as $\ker\kappa_R=\bigcap_{n\geq1} R_n=\{0\}$.
Thus, we obtain Eq.~\meqref{eq:reyhom}, and then
\begin{align*}
\bar{f}(\fraka\shpr\frakb)&\stackrel{\eqref{eq:shpr}}{=}
\bar{f}(a_0b_0\otimes(
\fraka'\,\hat{\ssha}\,\frakb'))=i_R(a_0b_0)P_R(\bar{f}(\fraka'\,\hat{\ssha}\,\frakb'))
\stackrel{\eqref{eq:reyhom}}{=}i_R(a_0)P_R(\bar{f}(\fraka'))i_R(b_0)P_R(\bar{f}(\frakb'))=\bar{f}(\fraka)\bar{f}(\frakb).
\end{align*}
In summary, we have shown that $\bar{f}$ is also an algebra homomorphism.

Finally, since $\bar{f}$ is required to be a \wtd Reynolds algebra homomorphism such that $i_R=\bar{f}\circ j_A$, we must have
\begin{align*}
\bar{f}(a_0\ot\cdots\ot a_n)&=\bar{f}\Big(a_0\shpr P_{\sshai(A)}(a_1\shpr P_{\sshai(A)}\big(\cdots P_{\sshai(A)}(a_n)\cdots)\big)\Big)\\
&=i_R(a_0)P_R\big(i_R(a_1)P_R(\cdots P_R(i_R(a_n))\cdots)\big),
\end{align*}
for $n\geq0$. Hence, such a map $\bar{f}$ is unique.
\end{proof}

\subsubsection{{\bf Step 4.} $\shai(A)$ as the free \ocmdra}
\mlabel{sss:fcdra}
We next apply Theorem~\ref{th:fcdr} to construct the free commutative \ocmdra over
a \ddiff algebra $(A,d)$. Define a linear operator $D_{\sshai(A)}$ on $\shai(A)$ by
\begin{equation}\mlabel{eq:ddoa}
	D_{\sshai(A)}(a_0)\coloneqq d(a_0),\quad D_{\sshai(A)}(\fraka)\coloneqq d(a_0)\ot\fraka'+a_0\shpr\fraka'-d(1_A)\shpr\fraka,
\end{equation}
for $\fraka=a_0\ot\fraka'\in A^{\ot n}$ with $n\geq2$.
Alternatively, we have
\[D_{\sshai(A)}(a_0\ot\cdots\ot a_k)=d(a_0)\ot a_1\ot\cdots\ot a_k+a_0a_1\ot a_2\ot\cdots\ot a_k-d(1_A)a_0\ot a_1\ot\cdots\ot a_k,\]
for $a_0,\dots,a_k\in A$ with $k\geq1$.

By definition we have 
\begin{equation}
D_{\sshai(A)}P_{\sshai(A)}=\id_{\sshai(A)}, \quad D_{\sshai(A)}\left(\shai(A)_{n+1}\right)\subseteq \shai(A)_n,\quad n\geq1.
\mlabel{eq:dcomp}
\end{equation}

Now we can give the statement and proof of our construction of free \ddiffrey algebras.

\begin{theorem}
Let $(A,d)$ be a \ddiff algebra. The sextuple $$\big(\shai(A),\shai(A)_n,\shpr,D_{\sshai(A)},P_{\sshai(A)},j_A\big)$$
is the free object in the category $\cdrac$ of \ocmdras over $(A,d)$.
\mlabel{thm:freedrey}
\end{theorem}

\begin{proof}
We first check that $D_{\sshai(A)}$ is a \ddiff operator on $\shai(A)$. Fix $\fraka=a_0\otimes\fraka'\in A^{\ot m}$ and $\frakb=b_0\ot\frakb'\in A^{\ot n},\,m,n\geq1$. We only deal with the case when $m,n>1$, since the case when $m=1$ or $n=1$ is simpler. Write $\fraka=a_0\otimes a_1\otimes\fraka''$ and $\fraka=b_0\otimes b_1\otimes\frakb''$ with $a_0,a_1,b_0,b_1\in A, \fraka''\in A^{\ot (m-1)}$ and $\frakb''\in A^{\ot (n-1)}$. We first derive
\begin{align*}
&\qquad a_0\shpr\fraka' \shpr \frakb+\fraka \shpr b_0\shpr \frakb'- d(1_A)\shpr \fraka \shpr \frakb\\
&\stackrel{\eqref{eq:shpr}}{=}
a_0b_0a_1\otimes(\fraka''\,\hat{\ssha}\,\frakb')+a_0b_0b_1\otimes(\fraka'\,\hat{\ssha}\,\frakb'')
-d(1_A)a_0b_0\otimes(\fraka'\,\hat{\ssha}\,\frakb')\\
&=
a_0b_0\shpr(a_1\otimes(\fraka''\,\hat{\ssha}\,\frakb')+b_1\otimes(\fraka'\,\hat{\ssha}\,\frakb'')
-d(1_A)\otimes(\fraka'\,\hat{\ssha}\,\frakb'))\\
&\stackrel{\eqref{eq:ast2}}{=}
a_0 b_0\shpr (\fraka'\,\hat{\ssha}\,\frakb').
\end{align*}
Then we have
\begin{align*}
D_{\sshai(A)}(\fraka \shpr \frakb)
&\stackrel{\eqref{eq:shpr}}{=}D_{\sshai(A)}(a_0b_0\ot(\fraka'\,\hat{\ssha}\,\frakb'))\\
&\stackrel{\eqref{eq:ddoa}}{=} d(a_0b_0)\ot(\fraka'\,\hat{\ssha}\,\frakb')+a_0b_0\shpr(\fraka'\,\hat{\ssha}\,\frakb')-d(1_A)\shpr(\fraka \shpr \frakb)\\
&=d(a_0b_0)\ot(\fraka'\,\hat{\ssha}\,\frakb')+a_0\shpr\fraka' \shpr \frakb+\fraka \shpr b_0\shpr \frakb'- 2d(1_A)\shpr \fraka \shpr \frakb.
\end{align*}
On the other hand,
\begin{align*}
&\qquad D_{\sshai(A)}(\fraka)\shpr \frakb + \fraka\shpr D_{\sshai(A)}(\frakb)- d(1_A)\shpr \fraka \shpr \frakb\\
&\stackrel{\eqref{eq:ddoa}}{=} (d(a_0)\ot\fraka'+a_0\shpr\fraka'-d(1_A)\shpr\fraka)\shpr \frakb
+\fraka\shpr (d(b_0)\ot\frakb'+b_0\shpr\frakb'-d(1_A)\shpr\frakb)- d(1_A)\shpr \fraka \shpr \frakb\\
&\stackrel{\eqref{eq:shpr}}{=} (d(a_0)b_0+a_0d(b_0)-d(1_A)a_0b_0)\ot(\fraka'\,\hat{\ssha}\,\frakb')+a_0\shpr\fraka' \shpr \frakb+\fraka \shpr b_0\shpr \frakb'- 2d(1_A)\shpr \fraka \shpr \frakb\\
&\stackrel{\eqref{Dda}}{=}d(a_0b_0)\ot(\fraka'\,\hat{\ssha}\,\frakb')+a_0\shpr\fraka' \shpr \frakb+\fraka \shpr b_0\shpr \frakb'- 2d(1_A)\shpr \fraka \shpr \frakb.
\end{align*}
Therefore, we have proved that Eq.~\eqref{Dda} holds for $D_{\sshai(A)}$.

Furthermore, given any \ocmdra $(R,R_n,D_R,P_R,i_R)$ (with a \ddiff algebra homomorphism $i_R:(A,d)\rar (R,D_R)$ as the structure map),
we just need to show that the induced homomorphism
$\bar{f}:\shai(A)\rar R$ of \wtd Reynolds algebras as defined in the proof of Theorem~\ref{th:fcdr} is moreover
a \ddiff algebra homomorphism. That is, $\bar{f}(D_{\sshai(A)}(\fraka)) = D_R(\bar{f}(\fraka))$ for all pure tensors $\fraka\in A^{\ot m}$. The case when $m=1$ follows from $i_Rd=D_Ri_R$. For $m\geq 2$, as $\bar{f}(\fraka)=i_R(a_0)P_R(\bar{f}(\fraka'))$, we have
\begin{align*}
\bar{f}(D_{\sshai(A)}(\fraka))&\stackrel{\eqref{eq:ddoa}}{=}\bar{f}(d(a_0)\ot\fraka'+a_0\shpr\fraka'-d(1_A)\shpr\fraka)\\
&=i_R(d(a_0))P_R(\bar{f}(\fraka')) +i_R(a_0)\bar{f}(\fraka')-i_R(d(1_A))\bar{f}(\fraka)\\
&=D_R(f(a_0))P_R(\bar{f}(\fraka')) +i_R(a_0)\bar{f}(\fraka')-D_R(1_R)\bar{f}(\fraka).
\end{align*}
On the other hand, 
\begin{align*}
D_R(\bar{f}(\fraka))&=D_R(i_R(a_0)P_R(\bar{f}(\fraka'))) \\
&\stackrel{\eqref{Dda}}{=}
D_R(i_R(a_0))P_R(\bar{f}(\fraka'))+i_R(a_0)D_R(P_R(\bar{f}(\fraka')))-D_R(1_R)i_R(a_0)P_R(\bar{f}(\fraka'))\\
&\stackrel{\eqref{Dra}}{=}D_R(i_R(a_0))P_R(\bar{f}(\fraka')) +i_R(a_0)\bar{f}(\fraka')-D_R(1_R)\bar{f}(\fraka).
\end{align*}
Hence, $\bar{f} D_{\sshai(A)} = D_R \bar{f}$. This completes the proof.
\end{proof}

We end the paper with applying Theorem~\mref{thm:freedrey} to the analytic setting in the previous section. Here the construction of the free \ocmdra can be made explicit.

\begin{prop}
Retaining the notions and assumptions of Proposition~\mref{pp:powerex}, if the coefficient of $x$ in $k(x)\in \RR[[x]]$ is nonzero, and $h(x)=\mu\frac{d}{dx}(1/k(x))$ for some $\mu\in\RR^*$, then the quintuple $(\RR[[x]],\RR[[x]]_n,D_K,P_K,i_{\RR[[x]]})$ is a \ocmdra over the \ddiff algebra $(\RR,\mu^{-1}\id_\RR)$ and is isomorphic to the free commutative \ocmdra\ $\shai_\RR(\RR, \mu^{-1}\id_\RR)$ over $(\RR, \mu^{-1}\id_\RR)$.
\mlabel{pp:powerex-free}
\end{prop}
Thus the \ocmdra $(\RR[[x]],\RR[[x]]_n,D_K,P_K,i_{\RR[[x]]})$ gives an analytic realization of the free commutative \ocmdra\ $\shai_\RR(\RR, \mu^{-1}\id_\RR)$.

\begin{proof}
If the coefficient of $x$ in $k(x)$ is nonzero, then $h(x)=\mu\frac{d}{dx}(1/k(x))$ is invertible. Thus,
 $$D_K(1)=\frac{1}{h(x)}\frac{d}{dx}\left(\frac{1}{k(x)}\right)=\mu^{-1}$$
and hence is in $\RR$. Then the restriction of $D_K$ to $\RR$ gives a \ddiff operator
$$d:=D_K|_{\RR}=\mu^{-1}\id_\RR$$
on $\RR$. Hence we can take $\RR$ to be the intermediate ring $A$ in Proposition~\mref{pp:powerex} and conclude that the quintuple $(\RR[[x]],\RR[[x]]_n,D_K,P_K,i_{\RR[[x]]})$ is a \ocmdra with $(A,d)=(\RR,\mu^{-1}\id_\RR)$.

Applying the universal property of the free commutative \ocmdra\ $\shai_\RR(\RR, \mu^{-1}\id_\RR)$ over $(\RR,\mu^{-1}\id_\RR)$, there is a homomorphism of \ddiffrey algebras
\[\phi:\shai_\RR(\RR,\mu^{-1}\id_\RR)\rar\RR[[x]],\ 1^{\otimes(n+1)}\mapsto P_K^n(1),\ n\geq0.\]
Next we prove by induction on $n$ that
\begin{equation} \mlabel{eq:pexpan}
P_K^n(1)=
\mu^n\dfrac{k(x)}{k(0)}\sum_{s\geq n}\dfrac{\ln^s(k(0)/k(x))}{s!},\ n\geq0.
\end{equation}
Here $\ln (k(0)/k(x))$ is regarded as a power series expansion.
For $n=0$, the equality is clear as
$$\dfrac{k(x)}{k(0)}\sum_{s\geq 0}\dfrac{\ln^s(k(0)/k(x))}{s!}=
\dfrac{k(x)}{k(0)}e^{\ln(k(0)/k(x))}=\dfrac{k(x)}{k(0)}\dfrac{k(0)}{k(x)}=1.$$
Now for $n>0$, we have
\begin{align*}
P_K^n(1)&=P_K(P_K^{n-1}(1))=\dfrac{\mu^{n-1}}{k(0)}\sum_{s\geq n-1}\dfrac{P_K(k(x)\ln^s(k(0)/k(x)))}{s!}\\
&=\dfrac{\mu^{n-1}}{k(0)}\sum_{s\geq n-1}\dfrac{k(x)}{s!}\int_0^xh(t)
k(t)\ln^s(k(0)/k(t))dt\\
&=\mu^n\dfrac{k(x)}{k(0)}\sum_{s\geq n-1}\dfrac{1}{s!}\int_0^x
\dfrac{d}{dt}\left(\dfrac{1}{k(t)}\right)k(t)\ln^s(k(0)/k(t))dt\\
&=\mu^n\dfrac{k(x)}{k(0)}\sum_{s\geq n-1}\dfrac{1}{s!}\int_0^x
\ln^s(k(0)/k(t))d(\ln(k(0)/k(t)))\\
&=\mu^n\dfrac{k(x)}{k(0)}\sum_{s\geq n-1}\dfrac{1}{(s+1)!}\ln^{s+1}(k(0)/k(t))\bigg|_0^x\\
&=\mu^n\dfrac{k(x)}{k(0)}\sum_{r\geq n}\dfrac{1}{r!}\ln^r(k(0)/k(x)).
\end{align*}
Since $h(0)=-\mu k'(0)/k^2(0)\in\RR^*$ by assumption, and the power series expansion of $k(0)/k(x)$ is $1-(k'(0)/k(0))x$ modulo higher degree terms, the expansion of $\ln(k(0)/k(x))$ has the lowest term of degree $1$.
Therefore, using Eq.\meqref{eq:pexpan} we see that
$$P_K^n(1)\in x^n\RR[[x]]\setminus x^{n+1}\RR[[x]].$$
As a result, $\phi$ is an isomorphism.
\end{proof}

Here is an application to the classical example, revisiting Remark~\mref{rk:rey}. 
\begin{exam}
Taking $K(x,t)=e^{-x+t}$ as in Example~\mref{ex:sepex} and Remark~\mref{rk:rey}, then $P_K$ is a Reynolds operator (of weight $1$), since $D_K(1)=1$. It is a special case of  Proposition~\mref{pp:powerex-free} with $\mu=1$.
Therefore, there is an isomorphism of \ocmdras over $(\RR,\id_\RR)$ 
\[\phi:\shai_\RR(\RR,\id_\RR)\rar\RR[[x]],\ 1^{\otimes(k+1)}\mapsto
P_K^k(1)=\sum_{n\geq k}(-1)^{n-k}{n-1 \choose k-1}\dfrac{x^n}{n!},\ k\geq0,\]
with the convention that ${i \choose -1}=\delta_{i,-1}$. So the natural (power series) product on the right hand side corresponds to the complete shuffle product on the left hand side. 
\mlabel{ex:rey2}
\end{exam}

The next example can be regarded as a degenerated case of Proposition~\mref{pp:powerex-free}, reducing to the usual derivation and integration.
\begin{exam}
Consider the \ddiffrey algebra $(\RR[[x]],D_K,P_K)$ with kernel $K(x,t)=1$. Then  $D_K=\tfrac{d}{dx}$ and $P_K=\int_0^xdt$.
Since $D_K(1)=0$, the pair $(\RR,0)$ is a \ddiff algebra and there is a \ddiff algebra monomorphism
$i_{\RR[[x]]}:(\RR,0)\to (\RR[[x]],D_K)$.
Take the decreasing filtration $\{x^n\RR[[x]]\}_{n\geq 0}$ of $\RR$-submodules of $\RR[[x]]$.
As discussed in Example~\mref{ex:comp}, we know that $(\RR[[x]],x^n\RR[[x]],D_K,P_K,i_{\RR[[x]]})$ is a \ocmdra with $(A,d)=(\RR,0)$.
Then the universal property of the free commutative \ocmdra $\shai_\RR(\RR, 0)$ over $(\RR,0)$ gives a homomorphism of \ocmras over $(\RR,0)$,
\[\psi:\shai_\RR(\RR,0)\rar\RR[[x]],\ 1^{\otimes(n+1)}\mapsto P_K^k(1)=\dfrac{x^n}{n!},\ n\geq0,\]
which is actually an isomorphism.
In particular, $\psi\Big(\sum\limits_{k\geq1}1^{\otimes k}\Big)=\sum\limits_{n\geq0}\tfrac{x^n}{n!}=e^x$.
\end{exam}

The last example goes beyond the scope of Proposition~\mref{pp:powerex-free}, but still falls into the range of discussion in Proposition~\mref{pp:powerex}.
\begin{exam}
Consider the rational kernel $K(x,t)=\frac{1}{x^2+1}$. We have
$D_K(x^k)=kx^{k-1}+(k+2)x^{k+1}$ for $k\geq 0$. Hence, $D_K$ restricts to a \ddiff operator $d$ on the polynomial algebra $\RR[x]$ defined by
$$d(x^k)=kx^{k-1}+(k+2)x^{k+1},\quad k\geq0.$$
By Proposition~\mref{pp:powerex}, there is a \ddiff algebra monomorphism
$i_{\RR[[x]]}:(\RR[x],d)\to (\RR[[x]],D_K)$, and
$(\RR[[x]],\RR[[x]]_n,D_K,P_K,i_{\RR[[x]]})$ is a \ocmdra over $(A,d)=(\RR[x],d)$.

Applying the universal property of the free commutative \ocmdra $\shai_\RR(\RR[x],d)$, there is a homomorphism of \ddiffrey algebras
\[\varphi:\shai_\RR(\RR[x],d)\rar\RR[[x]],\,x^{\ot \alpha}=x^{a_1}\ot \cdots \ot x^{a_r} \mapsto
Q_\alpha(x), \quad \alpha=(a_1,\ldots,a_r)\in \NN^r, \]
where  $Q_\alpha(x)$ is the following analytic function given by iterated integrals.
\begin{align*}
x^{a_0}P_K&\Big(x^{a_1}P_K\big(
\cdots x^{a_{r-2}}P_K(x^{a_{r-1}}P_K(x^{a_r}))\cdots\big)\Big)\\
&=\frac{x^{a_0}}{x^2+1}\int_0^x \frac{x_1^{a_1}}{x_1^2+1} dx_1\cdots \int_0^{x_{r-2}}\frac{x_{r-1}^{a_{r-1}}}{x_{r-1}^2+1}dx_{r-1}\int_0^{x_{r-1}}x_r^{a_r}dx_r.
\end{align*}
Applying the multiplication formula of the free objects in Eq.~\meqref{eq:astexp2}, a product of two iterated integrals of this form can be expressed as a series of iterated integrals of the same form.
\end{exam}

\smallskip

\noindent
{\bf Acknowledgments.} This work is supported by Natural Science Foundation of China (12071094, 12171155) and Guangdong Basic and Applied Basic Research Foundation (2022A1515010357).  The authors thank Markus Rosenkranz for helpful discussions. 

\noindent
{\bf Declaration of interests. } The authors have no conflicts of interest to disclose.

\smallskip

\noindent
{\bf Data availability. } No new data were created or analyzed in this study.

\end{document}